\documentclass[11pt, oneside]{amsart} 
\usepackage{amsmath, amscd}
\usepackage{amsfonts}
\usepackage{amssymb}
\usepackage{amsthm}
\usepackage{upref}
\usepackage{multirow}
\usepackage{graphicx}
\usepackage{subfloat}
\usepackage{subfig}
\usepackage{url}
\usepackage{mathtools}
\usepackage{pinlabel}
\usepackage{color}
\usepackage{hyperref}
\usepackage{changepage}
\usepackage[textwidth=6in, textheight=9in,marginpar=0.75in]{geometry}

\newcommand{\Z}{\mathbb{Z}}
\newcommand{\N}{\mathbb{N}}
\newcommand{\R}{\mathbb{R}}

\newcommand{\ra}{\rightarrow}

\newcommand{\gspan}[1]{\left\langle{#1}\right\rangle}

\newcommand{\emptyword}{\varepsilon}
\newcommand{\groupid}{1}
\newcommand{\Nf}{\ensuremath{\mathcal{N}}}
\def\nf#1{\mathsf{nf}(#1)}
\newcommand{\lbl}{{\mathsf{label}}}  
\newcommand{\Graph}[1]{\mathsf{graph}\left({#1}\right)}

\newcommand{\tree}{T}                       
\newcommand{\Tr}{\ensuremath{\mathcal{T}}}  
\newcommand{\NTr}{\ensuremath{\Nf_{\Tr}}}   
\newcommand{{\rsp}}{respecting}

\newcommand{\gog}{\ensuremath{\mathcal{G}}}
\newcommand{\homom}{h}                      

\newcommand{\tlam}{T_\Lambda}
\newcommand{\velam}{\vec E_\Lambda}
\newcommand{\velnot}{\vec E_{\Lambda \setminus T}}

\newcommand{\vst}{i}                         
\newcommand{\vend}{t}                        
\newcommand{\suf}{\mathsf{suf}}              
\newcommand{\edgeonly}{\rho}                 
\newcommand{\pt}{\mathsf{path}_{T_\Lambda}}  
\newcommand{\infl}{\mathsf{infl}}            
\newcommand{\defl}{\mathsf{defl}}            
\newcommand{\trim}{\mathsf{prune}}           
\newcommand{\epair}{\mathsf{last}}           
\newcommand{\counter}{\alpha}                
\newcommand{\dcl}{\mathsf{dcl}}              
\newcommand{\idcl}{\mathsf{idcl}}            

\newcommand{\hhh}{\mathsf{hat}}

\newcommand{\HH}{\ensuremath{\mathcal{H}}}
\newcommand{\last}{\mathsf{last}}
\newcommand{\nice}{nice}

\newtheorem{thm}{Theorem}[section]
\newtheorem{cor}[thm]{Corollary}
\newtheorem{lem}[thm]{Lemma}
\newtheorem{prop}[thm]{Proposition}

\theoremstyle{remark}

\theoremstyle{definition}
\newtheorem{defn}[thm]{Definition}
\newtheorem{rmk}[thm]{Remark}

\newcounter{claimcounter}[thm]
\numberwithin{claimcounter}{thm}
\newcounter{casecounter}

\newcounter{subcasecounter}[casecounter]
\numberwithin{subcasecounter}{casecounter}

\long\def\Restate#1#2#3#4{
\medskip\par\noindent
{\bf #1 \ref{#2} #3} {\it #4}\par\medskip }

\long\def\Restatesd#1#2#3#4#5#6{
\medskip\par\noindent
{\bf #1 \ref{#2} and #6 \ref{#5} #3} {\it #4}\par\medskip }

\begin{document}

\title{Geometry of the word problem for 3-manifold groups}
\author[M. Brittenham]{Mark Brittenham}
\address{Department of Mathematics, University of Nebraska, Lincoln, NE 68588-0130, USA}
\email{mbrittenham2@unl.edu}

\author[S. Hermiller]{Susan Hermiller}
\address{Department of Mathematics, University of Nebraska, Lincoln, NE 68588-0130, USA}
\email{hermiller@unl.edu}

\author[T. Susse]{Tim Susse}
\address{Department of Mathematics, Bard College at Simon's Rock,
  Great Barrington, MA 01230, USA}
\email{tsusse@simons-rock.edu}



\begin{abstract}
We provide an algorithm to solve the word problem in all fundamental
groups of  3-manifolds that are either closed, or compact with
(finitely many) boundary components consisting of incompressible
tori, by showing that these groups are autostackable.
In particular, this gives a common framework to solve the word
problem in these 3-manifold groups using finite state automata.

We also introduce the notion of a group which is
autostackable {\rsp} a subgroup, and show that a
fundamental group of a graph of groups
whose vertex groups are autostackable {\rsp} any edge group is
autostackable.  A group that is strongly coset automatic over
an autostackable subgroup, using a prefix-closed
transversal, is also shown to be autostackable
{\rsp} that subgroup.
Building on work by Antolin and Ciobanu, we
show that a finitely generated group that is hyperbolic relative
to a collection of abelian subgroups is also
strongly coset automatic relative to each subgroup
in the collection.
Finally, we show that fundamental groups of compact geometric
3-manifolds, with boundary consisting of
(finitely many) incompressible torus components,
are autostackable {\rsp} any choice of peripheral subgroup.
\end{abstract}

\maketitle


\section{Introduction}\label{sec:intro}


One fundamental goal in geometric group theory since its inception
has been to find algorithmic
and topological characteristics of the Cayley graph satisfied
by all closed 3-manifold fundamental
groups, to facilitate computations. This was an original
motivation for the definition of automatic groups by
Epstein, Cannon, Holt, Levy, Paterson, and Thurston~\cite{wordprocessing},
and its recent extension to Cayley automatic groups by
Kharlampovich, Khoussainov, and Miasnikov~\cite{KKM:Cayleyauto}.
These constructions, as well as finite convergent
rewriting systems, provide a solution to the word problem
using finite state automata.
However, automaticity fails
for 3-manifold groups in two of the eight geometries,
and Cayley automaticity and
finite convergent rewriting systems are
unknown for many 3-manifold groups.
Autostackable groups, first introduced by the
first two authors and Holt in~\cite{BHH:algorithms},
are a natural extension of both automatic groups and groups with
finite convergent rewriting systems. In common with these two
motivating properties, autostackability also gives a
solution to the word problem using finite state automata.
In this paper we show that
the fundamental group of every compact 3--manifold with
incompressible toral boundary, and hence every
closed 3-manifold group,  is autostackable. 




Let $G$ be a group with a finite
inverse-closed generating set $A$.
Autostackability is defined using a discrete
dynamical system on the Cayley graph
$\Gamma:=\Gamma_A(G)$ of $G$ over $A$, as follows.
A \emph{flow function} for $G$ with \emph{bound} $K \ge 0$,
with respect to
a spanning tree $T$ in $\Gamma$,
is a function $\Phi$ mapping the
set $\vec E$ of directed edges of $\Gamma$ to the
set $\vec P$ of directed paths in $\Gamma$, such that
\begin{description}
\item[(F1)] for each $e \in \vec E$ the path
$\Phi(e)$ has the same initial and terminal
vertices as $e$ and length at most $K$,
\item[(F2)] $\Phi$ acts as the identity on edges lying
in $T$ (ignoring direction),
and
\item[(F3)] there is no infinite sequence
$e_1,e_2,e_3,...$ of edges with each $e_i \in \vec E$
not in $T$ and each $e_{i+1}$ in 
the path $\Phi(e_i)$.
\end{description}
These three conditions are motivated by their
consequences for the extension
 $\widehat \Phi:\vec P \ra \vec P$ of $\Phi$ to
directed paths in $\Gamma$ defined by $\widehat
\Phi(e_1 \cdots e_n):=\Phi(e_1) \cdots \Phi(e_n)$, where $\cdot$
denotes concatenation of paths.  
Upon iteratively applying $\widehat \Phi$ to a path $p$,
whenever a subpath of $\widehat \Phi^n(p)$ lies in
$T$, then that subpath remains unchanged 
in any further iteration $\widehat \Phi^{n+k}(p)$, since
conditions (F1-2) show that 
$\widehat \Phi$ fixes any point
that lies in the tree $T$.
Condition (F3) 
ensures that 
for any path $p$ 
there is a natural
number $n_p$ such that $\widehat \Phi^{n_p}(p)$ is a path in
the tree $T$, and hence 
$\widehat \Phi^{n_p+k}(p) = \Phi^{n_p}(p)$ for all $k \ge 0$.  
The bound $K$ controls the 
extent to which each
application of $\widehat \Phi$ can alter a path. 
Thus when $\widehat \Phi$ is iterated, paths in
$\Gamma$ ``flow'', in bounded steps, toward the tree. 

A finitely generated group admitting
a bounded flow function over some finite set of generators is
called \emph{stackable}.
Let $\Nf_T$ denote the set of words
labeling non-backtracking paths in $T$ that start at the vertex
labeled by the identity $\groupid$ of $G$ (hence $\Nf_T$ is a
prefix-closed set of normal forms for $G$), and let $\lbl:\vec P \ra
A^*$ be the function that returns the label of any directed path in
$\Gamma$. The group $G$ is \emph{autostackable} if there is a finite
generating set $A$ with a bounded flow function $\Phi$ such that the
graph of $\Phi$, written in triples of strings over $A$ as
\begin{align*}
\Graph{\Phi}:=\{(y,a,\lbl(\Phi(e_{y,a})) \mid &
y \in \Nf_T, a \in A, \text{ and }
e_{y,a} \in \vec E
\text{ has } \\
& \text{ initial vertex } y \text{ and label }a\},
\end{align*}
is recognized by a finite state automaton
(that is, $\Graph{\Phi}$ is a (synchronously)
regular language).

To solve the word problem in an autostackable
group, given a word $w$ in $A^*$, by using the
finite state automaton recognizing $\Graph{\Phi}$
to iteratively replace
any prefix of the form $ya$ with $y \in \Nf_T$
and $a \in A$ by $y\lbl(\Phi(e_{y,a}))$
(when $\lbl(\Phi(e_{y,a}))$ is not $a$), and performing free
reductions, a word $w' \in \Nf_T$ is obtained,
and $w=_G \groupid$ if and only if
$w'$ is the empty word.
Hence autostackability also implies that the group
has a finite presentation.

Both autostackability and its motivating property
of automaticity also have equivalent definitions
as rewriting systems.  In particular, 
a group $G$ is autostackable if and only if 
$G$ admits a bounded regular convergent
prefix-rewriting system~\cite{BHH:algorithms},
and a group $G$ is automatic 
with prefix-closed normal forms if and only if 
$G$ admits an interreduced regular 
convergent prefix-rewriting system~\cite{otto}.
(See Section~\ref{sub:autostack} for definitions
of rewriting systems and Section~\ref{sub:cosetaut}
for a geometric definition of automaticity.)

The class of autostackable groups contains all
groups with a finite convergent rewriting system or an asynchronously
automatic structure with prefix-closed (unique) normal forms~\cite{BHH:algorithms}.
Beyond these examples, autostackable groups include some groups that
do not have homological type $FP_{3}$ \cite[Corollary~4.2]{BHJ:closure}
and some groups whose Dehn function is non-elementary primitive recursive;
in particular, Hermiller and Mart\'{i}nez-P\'erez show
in~\cite{HermillerMartinez:HNN} that the Baumslag-Gersten group
is autostackable.

We focus here on the case where $G$ is the fundamental group of a
connected, compact 3-manifold $M$ with incompressible toral boundary.
In~\cite{wordprocessing} it is shown that if
no prime factor of $M$ admits \emph{Nil} or
\emph{Sol} geometry, then $\pi_1(M)$ is automatic. However, the
fundamental group of any $\emph{Nil}$ or $\emph{Sol}$ manifold does
not admit an automatic, or even asynchronously automatic,
structure~\cite{wordprocessing, Brady:Sol}. Replacing the finite
state automata by automata with unbounded memory, Bridson and Gilman
show in~\cite{BridsonGilman:Indexed} that the group $G$
is asynchronously combable by an indexed language (that is, a
set of words recognized by a nested stack automaton), although for
some 3-manifolds the language cannot be improved to context-free
(and a push-down automaton). Another extension of automaticity,
solving the word problem with finite state automata whose alphabets
are not based upon a generating set, is given by the more recent
concept of Cayley graph automatic groups, introduced by
Kharlampovich, Khoussainov and Miasnikov in \cite{KKM:Cayleyauto};
however, it is an open question whether all fundamental groups of
closed 3-manifolds with \emph{Nil} or \emph{Sol} geometry are Cayley
graph automatic. From the rewriting viewpoint,
in~\cite{HermillerShapiro:rewriting} Hermiller and Shapiro showed
that fundamental groups of closed fibered hyperbolic 3-manifolds
admit finite convergent rewriting systems, and that all closed
geometric 3-manifold groups in the other 7 geometries do as well.
However, the question of whether all closed 3-manifold groups admit
a finite convergent rewriting system also remains open.


In this paper we show that \emph{every} fundamental group of a
connected, compact 3-manifold with incompressible toral boundary is
autostackable. The results of~\cite{BHH:algorithms} above show that
the fundamental group of any closed geometric 3-manifold is
autostackable; here, we will show that the restriction to geometric
manifolds is unnecessary. To do this, we investigate the
autostackability of geometric pieces arising in the JSJ
decomposition of a 3-manifold, along with closure properties of
autostackability under the construction of fundamental groups of
graphs of groups, including amalgamated products and HNN extensions.

We begin with background on
automata, autostackability, rewriting systems,
fundamental groups of graphs of groups,
strongly coset automatic groups,
relatively hyperbolic groups, and 3-manifolds
in Section~\ref{sec:background}.

Section~\ref{sec:autstkgog} contains the proof
of the autostackability closure property for graphs of groups.
We define a group $G$ to be \emph{autostackable {\rsp}}
a finitely generated subgroup $H$ if
$G$ has an autostackable structure with flow function
$\Phi$ and spanning tree $T_G$
on a generating set $A$ satisfying:
\begin{itemize}
\item[]{\em Subgroup closure:}
There is a finite inverse-closed generating set $B$ for $H$
contained in $A$ such that
$T_G$ contains a spanning tree $T_H$
for the subgraph $\Gamma_B(H)$ of $\Gamma_A(G)$,
and for all $h \in H$ and $b \in B$,
$\lbl(\Phi(e_{h,b})) \in B^*$.
\item[]{\em $H$--translation invariance:}
There is a subtree $T'$ of $T_G$ containing
the vertex $\groupid$ such that the left action
of $H$ on $\Gamma_G(A)$ gives 
$T_G = T_H \cup (\cup_{h \in H} hT')$,
and for all $h,\tilde h \in H$ the trees
$hT', \tilde h T'$ are disjoint and
the intersection of the trees $hT', T_H$
is the vertex $h$. 
Moreover, the group action outside of $\Gamma_B(H)$
preserves the label of the flow function; that is,
for all directed edges
$e_{g,a}$ of $\Gamma_A(G)$ not in
$\Gamma_B(H)$ 
(with $g \in G$ and $a \in A$)
and for all $h \in H$, 
the flow function satisfies
$\lbl(\Phi(e_{g,a}))=
\lbl(\Phi(e_{hg,a}))$.
\end{itemize}
(As above, $e_{g,a}$ denotes the directed edge of $\Gamma_A(G)$
with initial vertex $g$ and label $a$.)
The conditions on the tree $T_G$
are equivalent to the requirement that the associated
normal form set $\Nf_G$ satisfy
$\Nf_{G} = \Nf_{H}\Nf_{\Tr}$
for some prefix-closed sets $\Nf_{H} \subset B^*$ of normal
forms for $H$ and
$\Nf_{\Tr} \subset A^*$ of normal
forms for the set of right cosets $H \backslash G$.
The subgroup closure condition together with 
Lemma~\ref{lemma:respnf}
imply that the subgroup $H$
is also autostackable.
If the requirement that the graph
of the flow function is a regular language is removed,
we say that $G$ is \emph{stackable {\rsp}} $H$.
We show that autostackability of vertex
groups {\rsp} edge groups suffices to preserve
autostackability for graphs of groups.

\Restate{Theorem}{thm:GoGautostack}{}
{Let $\gog$ be a graph of groups
over a finite connected graph $\Lambda$ with at least one edge.
If for each directed edge $e$ of $\Lambda$
the vertex group $G_v$ corresponding
to the terminal vertex $v=t(e)$ of $e$
is autostackable [respectively, stackable] {\rsp}
the associated injective homomorphic image
of the edge group $G_e$, then the
fundamental group
$\pi_1(\gog)$ is autostackable [respectively, stackable].
}

We note that for the two word problem
algorithms that motivated autostackability,
some closure properties for the graph of groups
construction have been found, but with other added restrictions.
For automatic groups, closure
of amalgamated free products and HNN extensions
over finite subgroups is shown
in~\cite[Thms~12.1.4,~12.1.9]{wordprocessing}
and closure for amalgamated products
under other restrictive geometric and language theoretic
conditions has been shown in~\cite{BGGS:autoamalgams}.
For groups with finite convergent
rewriting systems, closure for HNN-extensions
in which one of the associated subgroups equals the base
group and the other has finite index in the base group
is given in~\cite{GS}.
Closure for stackable groups in the special case of an HNN extension
under significantly relaxed  assumptions
(and using left cosets instead of right)
are given by the second author and Mart\'{i}nez-P\'erez in
\cite{HermillerMartinez:HNN}.  They also prove a closure result
for HNN extensions of autostackable groups, with a
requirement of further technical assumptions.

Section~\ref{sec:extn} contains a discussion of extensions of
two autostackability closure results of~\cite{BHJ:closure}
to autostackability {\rsp} subgroups, namely
for extensions of groups and finite index supergroups.

In Section~\ref{sec:cosetautgp} we study the
relationship between autostackability of a
group $G$ {\rsp} a subgroup $H$ and strong
coset automaticity of $G$ with respect
to $H$ defined by Redfern~\cite{Redfern:thesis}
and Holt and Hurt~\cite{HoltHurt:Coset}
(referred to as coset automaticity
with the coset fellow-traveler property
in the latter paper; see Section~\ref{sub:cosetaut} below for
definitions).
More precisely, we prove the following.

\Restate{Theorem}{thm:cosetautostack}{}
{
Let $G$ be a finitely generated group
and $H$ a finitely generated autostackable subgroup of $G$.
If the pair $(G,H)$ is strongly prefix-closed coset automatic,
then $G$ is autostackable {\rsp} $H$.}

Applying this in the case where $G$ is hyperbolic relative to
a collection of sufficiently nice subgroups, and building upon
work of Antolin and Ciobanu~\cite{AntolinCiobanu:relhyp},
we obtain the following.

\Restate{Theorem}{thm:relhypauto}{}
{Let $G$ be a group 
that is hyperbolic relative to a collection of subgroups $\{H_1,...,H_n\}$
and is generated by a finite set $A'$.
Suppose that for every index $j$, the group
$H_j$ is shortlex biautomatic on every finite ordered generating set.
Then there is a finite subset
$\HH' \subseteq \HH:=\cup_{j=1}^n (H_j \setminus \groupid)$ such that
for every finite generating set $A$ of $G$ with
$A' \cup \HH' \subseteq A \subseteq A' \cup \HH$ and any ordering on $A$,
and for any $1 \le j \le n$, the pair $(G,H_j)$ is strongly shortlex
coset automatic, and $G$ is autostackable {\rsp} $H_j$, over $A$.}

In particular, if $G$ is hyperbolic relative to abelian subgroups,
then $G$ is autostackable {\rsp} any peripheral subgroup.

In Section 6 we prove our results on autostackability of 3-manifold
groups.  We begin by considering compact geometric 3-manifolds with
boundary consisting of a finite number of incompressible tori that
arise in a JSJ decomposition of a compact, orientable, prime
3-manifold. Considering the Seifert fibered and hyperbolic cases
separately,
we obtain the following. 

\Restatesd{Proposition}{prop:seifertautostack}{}
{Let $M$ be a finite volume geometric 3-manifold with
incompressible toral boundary. Then
for each choice of component $T$ of $\partial M$,
the group $\pi_1(M)$ is autostackable {\rsp} any conjugate of
$\pi_1(T)$.}{prop:relhypautostack}{Corollary}

In comparison, finite convergent
rewriting systems have been found
for all fundamental groups of Seifert fibered knot complements,
namely the torus knot groups, by Dekov~\cite{dekov},
and for fundamental groups of alternating knot complements,
by Chouraqui~\cite{Chouraqui:FCRS}.
In the case of a finite volume hyperbolic
3-manifold $M$, the fundamental group
$\pi_1(M)$ is hyperbolic relative to the
collection of fundamental groups of its
torus boundary components by a result of Farb~\cite{Farb:relhyp},
and so by closure of the class of
(prefix-closed) biautomatic
groups with respect to relative hyperbolicity
(shown by Rebbecchi in~\cite{rebbechi};
see also~\cite{AntolinCiobanu:relhyp}), the group
$\pi_1(M)$ is biautomatic.

Combining this result on fundamental groups of pieces arising from
JSJ decompositions with Theorem~\ref{thm:GoGautostack}, together
with other closure properties for autostackability, yields the
result on 3-manifold groups.

\Restate{Theorem}{thm:3mfldauto}{} {Let $M$ be a compact
$3$-manifold with incompressible toral boundary. Then $\pi_1(M)$ is
autostackable. In particular, if $M$ is closed, then $\pi_1(M)$ is
autostackable.}



\subsection*{Acknowledgments}


The second author was partially supported by a grant from
the National Science Foundation (DMS-1313559).


\section{Background: Definitions and notation}\label{sec:background}


In this section we assemble definitions, notation, and theorems 
that will be used in the rest of the paper, in order
to make the paper more self-contained.

Let $G=\gspan{A}$ be a group.
Throughout this paper we will assume that every
generating set is finite and inverse-closed,
and every generating set for a flow function
does not contain a letter representing the
identity element of the group.
By $``="$ we mean equality in $A^*$,
while $``=_G"$ denotes equality in the group $G$.
For a word $w\in A^*$, we denote its length by $\ell(w)$.
The identity of the group $G$ is written $\groupid$ and
the empty word in $A^*$ is $\emptyword$.

Let $\Gamma_A(G)$ be the Cayley graph of $G$ with the generators
$A$. We denote by $\vec E=\vec{E}_A(G)$ the set of
oriented edges of
the Cayley graph, and denote by $\vec P=\vec{P}_A(G)$
the set of directed edge paths in $\Gamma_A(G)$.
By $e_{g,a}$ we mean the oriented edge with
initial vertex $g$ labeled by $a$.

By a set of
\emph{normal forms} we mean the image $\Nf=\sigma(G) \subset A^*$
of a section $\sigma\colon G\to A^*$
of the natural monoid homomorphism $A^*\to G$. In particular, every
element of $G$ has a unique normal form.
For $g\in G$, we denote its
normal form by $\nf{g}$.

Let $H\le G$ be a subgroup. By a
right coset of $H$ in $G$ we mean a subset of the form $Hg$ for
$g\in G$. We denote the set of right cosets by $H \backslash G$.
A subset $\Tr\subseteq G$ is a \emph{right transversal}
for $H$ in $G$ if every right coset of $H$ in $G$ has a
unique representative in $\Tr$.


\subsection{Regular languages}\label{sub:reg}



$~$

\medskip

A comprehensive reference on the contents of this section
can be found in~\cite{wordprocessing,HopcroftUllman}; see
\cite{BHH:algorithms} for a more concise introduction.

Let $A$ be a finite set, called an \emph{alphabet}.
The set of all finite strings over $A$
(including the empty word $\emptyword$)
is written $A^*$.
A \emph{language} is a subset
$L\subseteq A^*$. Given
languages $L_1,L_2$
the \emph{concatenation} $L_1L_2$
of $L_1$ and $L_2$ is
the set of all expressions of the form $l_1l_2$ with $l_i\in L_i$.
Thus $A^k$ is the set of all words of length $k$ over $A$;
similarly,  we denote the set of all words of length at most
$k$ over $A$ by $A^{\le k}$.
The \emph{Kleene star} of $L$, denoted $L^*$,
is the union of $L^n$ over all integers $n\ge 0$.

The class of \emph{regular languages}
over $A$ is the smallest class of languages
that contains all finite languages and is closed under union,
intersection, concatenation, complement and Kleene star. (Note that
closure under some of these operations is redundant.)

Regular languages are precisely those accepted by finite state
automata; that is, by computers with a bounded amount of memory.
More precisely, a finite state automaton consists of a finite set of
states $Q$, an initial state $q_0 \in Q$, a set of accept states $P
\subseteq Q$, a finite set of letters $A$, and a transition function
$\delta:Q \times A \ra Q$.  The map $\delta$ extends to a function
$\delta:Q \times A^* \ra Q$; for a word $w=a_1 \cdots a_k$ with each
$a_i$ in $A$, the transition function gives
$\delta(q,w)=\delta(\cdots (\delta(\delta(q,a_1),a_2),\cdots,a_k)$.
The automaton can also considered as a directed labeled graph whose
vertices correspond to the state set $Q$, with a
directed edge from $q$ to $\delta(q,a)$ labeled by $a$
for each $a \in A$ and $q \in Q$. Using this
model $\delta(q,w)$ is the terminal vertex of the path starting at
$q$ labeled by $w$. A word $w$ is in the language of this automaton
if and only if $\delta(q_0,w) \in P$.

The concept of regularity is extended to
subsets of a Cartesian product
$(A^*)^n=A^* \times \cdots \times A^*$ of $n$ copies
of $A^*$ as follows.
Let $\$$ be a symbol not contained in $A$.
Given any tuple $w=(a_{1,1} \cdots a_{1,m_1},...,
a_{n,1} \cdots a_{n,m_n}) \in (A^*)^n$ (with
each $a_{i,j} \in A$), rewrite
$w$ to a \emph{padded word} $\hat w$ over the finite
alphabet $B:=(A \cup \$)^n$ by
$\hat w :=
(\hat a_{1,1},...,\hat a_{n,1}) \cdots (\hat a_{1,N},...,\hat a_{n,N})$
where $N=\max\{m_i\}$ and $\hat a_{i,j} =a_{i,j}$
for all $1 \le i \le n$
and $1 \le j \le m_i$ and $\hat a_{i,j}=\$$ otherwise.
A subset $L \subseteq (A^*)^n$ is called
a \emph{regular language} (or, more precisely,
\emph{synchronously regular}) if the set
$\{ \hat w \mid w \in L\}$ is a regular subset
of  $B^*$.

The following theorem, much of the proof of which can be found
in \cite[Chapter 1]{wordprocessing}, contains  closure
properties of regular languages that are used later in this paper.

\begin{thm}\label{thm:regclosure}
Let $A,B$ be finite alphabets, $x$ an element of $A^*$,
$L,L_i$ regular languages over $A$,
$K$  a regular language over $B$,
$\phi\colon A^*\to B^*$ a monoid homomorphism,
$L'$ a regular subset of $(A^*)^n$, and
$p_i:(A^*)^n \ra A^*$ the projection map on the
$i$-th coordinate.
Then the following languages are
also regular:
\begin{enumerate}
\item (Homomorphic image) $\phi(L)$.
\item (Homomorphic preimage) $\phi^{-1}(K)$.
\item (Quotient) $L_x:=\{w\in A^*: wx\in L\}$.
\item (Product) $L_1 \times L_2 \times \cdots \times L_n$.
\item (Projection) $p_i(L)$.
\end{enumerate}
\end{thm}


\subsection{Autostackability and rewriting systems}\label{sub:autostack}


$~$

\medskip

Proofs of the results in this
section and more detailed background on autostackability
are in~\cite{BHH:algorithms, BHJ:closure, HermillerMartinez:HNN}.
Let $G=\gspan{A}$ be an autostackable group,
with spanning tree $T$ in $\Gamma_A(G)$ and
flow function $\Phi:\vec E \ra \vec P$.

As noted in Section~\ref{sec:intro}, the tree $T$
defines a set of prefix-closed normal forms, denoted $\Nf=\Nf_T$,
for $G$, namely the words that label non-backtracking paths
in $T$ with initial vertex $1$; as above, we denote the normal form
of $g \in G$ by $\nf{g}$. Since $\Nf=p_1(\Graph{\Phi})$,
where $p_1$ denotes projection on the first coordinate,
Theorem~\ref{thm:regclosure} implies that the set $\Nf$
is a regular language over $A$.

An illustration of the flow function path associated
to the edge $e_{g,a}$ in the Cayley graph $\Gamma_A(G)$
is given in Figure~\ref{fig:flow}.
\begin{figure}
\begin{center}
\includegraphics[width=3.8in,height=1.0in]{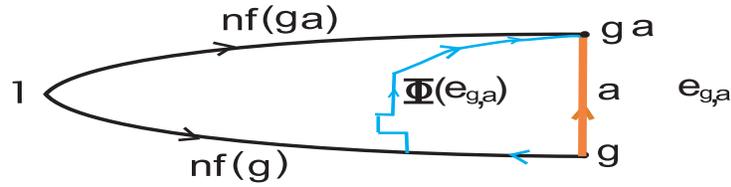}
\caption{The flow function}\label{fig:flow}
\end{center}
\end{figure}
The flow function $\Phi$ yields an algorithm to build a
van Kampen diagram for any word $w$ over $A$ representing
the trivial element of $G$.  Writing $w=a_1 \cdots a_n$,
diagrams for each of the words
$\nf{a_1 \cdots a_{i-1}}a_i\nf{a_1 \cdots a_i}^{-1}$ 
are recursively constructed, and then glued along 
the normal forms.
(Prefix closure of the normal forms implies that
each path labeled by a
normal form is a simple path, and hence this
gluing preserves planarity of the diagram.)
In particular, for $g \in G$ and $a \in A$, 
the edge $e_{g,a}$ lies in the tree $T$ if
and only if either $\nf{g}a=\nf{ga}$ or $\nf{ga}a^{-1}=\nf{g}$,
which in turn holds if and only if there is a
degenerate van Kampen diagram (i.e., containing no
2-cells) for the word $\nf{g}a\nf{ga}^{-1}$.
In the case that $e_{g,a}$ is not in $T$,
the diagram for $\nf{g}a\nf{ga}^{-1}$ is built
recursively from the 2-cell bounded by $e_{g,a}$
and $\Phi(e_{g,a})$ together with van Kampen diagrams for
the edges in the path $\Phi(e_{g,a})$.
See~\cite{BH:uniform,BHH:algorithms} for more details.

Since directed edges in $\Gamma_A(G)$ are in bijection
with $\Nf \times A$ and the set of paths in the Cayley graph based at
$g$ is in bijection with the set
$A^*$ of their edge labels, the flow function $\Phi$
gives the same information as
the \emph{stacking function} $\phi\colon \Nf\times A\to A^{\le K}$
defined by $\phi(\nf{g},a):=\lbl(\Phi(e_{g,a}))$.
Thus the set
$$
\Graph{\Phi}=\{(\nf{g},a,\phi(\nf{g},a)): g \in G, a \in A\}
$$
is the graph of this stacking function.
This perspective will be very useful in writing
down the constructions of flow functions throughout
the paper.

In \cite{BH:uniform} the first two authors show that if $G$ is a
stackable group whose flow function bound is $K$,
then $G$ is finitely presented with relators given
by the labels of loops in $\Gamma_A(G)$ of length at most $K+1$
(namely the relations $\phi(\nf{g},a)=_G a$). Although it is unknown
if autostackability is invariant under changes in finite generating
sets, we note that it is straightforward to show that if $G$ is
autostackable with generating set $A$ and $A\subseteq B$, then $G$
is autostackable with the generators $B$ using the same set of
normal forms (see~\cite[Proposition~4.3]{HermillerMartinez:HNN} for
complete details).

Not every
stackable group has decidable word problem;
Hermiller and Mart\'{i}nez-P\'erez~\cite{HermillerMartinez:HNN}
show that there
exist groups with a bounded flow function
but unsolvable word problem.
A group with a bounded flow function $\Phi$
whose graph is a recursive (i.e., decidable)
language has a word problem solution using
the automaton that recognizes $\Graph{\Phi}$~\cite{BHH:algorithms}.
Thus autostackable groups have word problem solutions using
finite state automata.

Autostackability also has an interpretation in terms
of prefix-rewriting systems.
A \emph{convergent prefix-rewriting system} for a group $G$ consists
of a finite set $A$ and a subset $R\subseteq A^*\times A^*$ such that
$G$ is presented as a monoid by
$\gspan{A\mid \{uv^{-1}:(u,v)\in R\}}$ and the rewriting operations of
the form $uz\to vz$ for all $(u,v)\in R$ and $z\in A^*$ satisfy:
\begin{itemize}
\item[]\emph{Termination.} There is no infinite sequence of rewritings
$x\to x_1\to x_2\to\ldots$
\item[]\emph{Normal Forms.} Each $g\in G$ is represented by a unique
irreducible word ({i.e.}, one that cannot be rewritten) over
$A$.
\end{itemize}
A prefix-rewriting system is called \emph{regular} if $R$ is a
regular subset of $A^* \times A^*$, and
$R$ is called
\emph{interreduced} if for each $(u,v) \in R$
the word $v$ is irreducible over $R$ and the word
$u$ is irreducible over $R \setminus \{(u,v)\}$.
A prefix-rewriting system $R$ is called \emph{bounded} if
there is a constant $K>0$ such that for each $(u,v)\in R$ there are
words $s,t,w\in A^*$ such that $u=ws$, $v=wt$ and
$\ell(s),\ell(t)\le K$.

The following is a recharacterization of
autostackability which is useful in interpreting the results
and proofs later in this paper.

\begin{thm}[Brittenham, Hermiller, Holt~\cite{BHH:algorithms}]
\label{thm:rewritingsystem} Let $G$ be a finitely generated group.
\begin{enumerate}
\item The group $G$ is autostackable if and only if $G$
admits a regular bounded convergent prefix-rewriting system.
\item The group $G$ is stackable if and only if $G$ admits a bounded
convergent prefix-rewriting system.
\end{enumerate}
\end{thm}

A \emph{finite convergent rewriting system} for $G$ consists of a
finite set $A$ and a finite subset $R' \subset A^* \times A^*$
presenting $G$ as a monoid such that the regular bounded
prefix-rewriting system $R:=\{(yu,yv) \mid y \in A^*, (u,v) \in
R'\}$ is convergent. Thus, Theorem~\ref{thm:rewritingsystem} shows
that autostackability is a natural extension of finite convergent
rewriting systems, in which the choice of direction of rewritings of
bounded length subwords depends upon the prefix appearing before the
subword to be rewritten.



\subsection{Automatic groups and coset automaticity}\label{sub:cosetaut}


$~$

\medskip

In~\cite{Redfern:thesis}, Redfern introduced the notion of coset
automatic group, as well as the geometric condition of strong coset
automaticity (using different terminology), studied in more detail
by Holt and Hurt in~\cite{HoltHurt:Coset}.  We consider strong coset
automaticity in this paper.

Let $G$ be a group generated by a finite inverse-closed set $A$
and let $\Gamma=\Gamma_A(G)$ be the corresponding Cayley graph.
Let $d_\Gamma$ denote the path metric distance in $\Gamma$.
For any word $v$ in $A^*$ and integer $i \ge 0$,
let $v(i)$ denote the prefix of $v$ of length $i$; that is,
if $v=a_1 \cdots a_m$ with each $a_j \in A$, then
$v(i):=a_1 \cdots a_i$ if $i \le m$ and $v(i)=v$ if
$i \ge m$.

\begin{defn}\cite{HoltHurt:Coset}\label{def:cosetaut}
Let $G$ be a group, let $H$ be a subgroup of $G$, and let $K \ge 1$
be a constant. The pair $(G,H)$ satisfies the \emph{$H$--coset
$K$--fellow traveler property} if there exists a finite
inverse-closed generating set $A$ of $G$ and a language $L \subset
A^*$ containing a representative for each right coset $Hg$ of $H$ in
$G$, together with a constant $K \ge 0$, satisfying the property
that for any two words $v,w \in L$ with $d_{\Gamma_A(G)}(v,hw) \le
1$ for some $h \in H$, we have $d_{\Gamma_A(G)}(v(i), hw(i)) \le K$
for all $i \ge 0$. The pair $(G,H)$ is \emph{strongly coset
automatic} if there exists a finite inverse-closed generating set
$A$ of $G$ and a regular language containing a representative for
each right coset of $H$ in $G$
satisfying the $H$--coset $K$--fellow traveler property for some $K
\ge 0$.

The pair $(G,H)$ is \emph{strongly prefix-closed coset automatic}
if the language $L$ is also prefix-closed and contains exactly
one representative of each right coset.  Given a total
ordering on $A$, the pair $(G,H)$ is \emph{strongly shortlex coset automatic}
if in addition $L$ contains only the shortlex least representative
(using the shortlex ordering induced by the ordering on $A$)
of 
each coset.
\end{defn}

A group $G$ is \emph{automatic} if the pair $(G,\{\groupid\})$ is
strongly coset automatic. Prefix-closed automaticity and shortlex
automaticity are obtained similarly.

We also consider the 2-sided notions of fellow traveling
and automaticity in Sections~\ref{sub:relhyp} and~\ref{sec:cosetautgp}.
The group $G$ is \emph{biautomatic} if there is a regular
language $L \subset A^*$
containing a representative of each element of $G$
and a constant $K \ge 0$ such that
for any $v,w \in L$ and
$a \in A \cup \{\emptyword\}$ with $d_{\Gamma_A(G)}(av,w) \le 1$,
we have $d_{\Gamma_A(G)}(\tilde v(i), w(i)) \le K$ for all $i \ge 0$,
where $\tilde v:=av$.  The adjective \emph{shortlex} is added if
$L$ is the set of shortlex least representatives of
the elements of $G$.

Holt and Hurt show that in the case that
the language $L$ is the shortlex transversal
(i.e., the shortlex least representatives for the cosets),
the coset fellow-traveler property
yields regularity of the language.

\begin{thm}\cite[Theorem~2.1]{HoltHurt:Coset}\label{thm:shortlexcosetaut}
Let $H$ be a subgroup of $G$, and let $A$ be a finite inverse-closed
totally ordered generating set for $G$. If $(G,H)$ satisfies the
$H$--coset $K$--fellow traveler property over $A$ using the
shortlex transversal for the right cosets of $H$ in $G$, then the
pair $(G,H)$ is strongly shortlex coset automatic.
\end{thm}

In their work on strong coset automaticity,
Holt and Hurt also describe
finite state automata that perform multiplication by
a generator in strongly shortlex coset automatic groups.
As we note in the next Proposition,
their construction works without
the shortlex ordering as well.
We provide some of the details of their construction
(with a slight modification),
in order to use them later
in the proof of Theorem~\ref{thm:cosetautostack}.
As above, $d_{\Gamma_A(G)}$ denotes the path metric distance
in the Cayley graph $\Gamma_A(G)$.  For any
radius $r \ge 0$, let
$$
B_{\Gamma_A(G)}(r) := \{g \in G \mid d_{\Gamma_A(G)}(\groupid,g) \le r\},
$$
be the set of vertices in the closed ball of radius $r$ in the Cayley graph.

\begin{prop}~\cite{HoltHurt:Coset}\label{prop:strongimpliesweak}
Let $H$ be a subgroup of a group $G$, and suppose that $(G,H)$ is
strongly coset automatic over a generating set $A$ of $G$ with
$H$--coset $K$--fellow traveling regular language $L$ of
representatives of the right cosets of $H$ in $G$. Then for each $h
\in H \cap B_{\Gamma_A(G)}(K)$
and $a \in A$, there is a finite state automaton $M_{h,a}$
accepting the set of padded words
corresponding to the set of word pairs
$$
L_{h,a}:=\{(x,y) \mid x,y \in L \text{ and } xa =_G hy\}.
$$
\end{prop}

\begin{proof}
Regularity of the set $L$ together with closure
of regular languages under product
(Theorem~\ref{thm:regclosure}) implies that the language
$L \times L \subset A^* \times A^*$ is also regular.
Hence the set of padded words corresponding to
the pairs of words in $L \times L$ is accepted by
a finite state automaton $M$, with state set $Q$,
initial state $q_0$, accept states $P$, alphabet
$(A \cup \$)^2$, and transition function
$\delta:Q \times A \ra Q$.

Note that the $H$--coset $K$--fellow traveler property implies that
for all $(x,y)$ in $L_{h,a}$, we have $x(i)^{-1}hy(i) \in
B_{\Gamma_A(G)}(K)$ for all $i \ge 0$, and so we can also write
$$L_{h,a}=\{(x,y) \mid x,y \in A^*, x(i)^{-1}hy(i) \in B_{\Gamma_A(G)}(K)
   \text{ for all } i \ge 0,
   \text{ and } xa =_G hy\}.
$$

We construct a finite state automaton $M_{h,a}$ as follows. The set
of states of $M_{h,a}$ is $\widetilde Q:=(Q \times
B_{\Gamma_A(G)}(K)) \cup \{F\}$, the initial state is $\widetilde
q_0:=(q_0,h)$, the set of accept states is $\widetilde P:=P \times
\{a\}$, and the alphabet is $(A \cup \$)^2$.  The transition
function $\widetilde \delta:\widetilde Q \times A \ra \widetilde Q$
is defined by $\widetilde
\delta((q,g),(a,b)):=(\delta(q,a),a^{-1}gb)$ if $a^{-1}gb \in
B_{\Gamma_A(G)}(K)$ and $\widetilde \delta((q,g),(a,b)):=F$
otherwise, and $\widetilde \delta(F,(a,b)):=F$ (here if either $a$
or $b$ is $\$$, it is treated as the group identity in the
expression $a^{-1}gb$). The language of $M_{h,a}$ is $L_{h,a}$.
\end{proof}


\subsection{Graphs of groups}\label{sub:GoG}


$~$

\medskip

A general reference to the material in this section,
with an algebraic approach together with proofs of basic facts about
graphs of groups (e.g., invariance under change of
spanning tree, injectivity of the natural inclusion of $G_v$
and $G_e$,  existence of the Bass-Serre tree, etc.),
can be found in~\cite{Serre:trees}.
A more topological viewpoint on this
topic is given in~\cite{ScottWall}.

Let $\Lambda$ be a connected graph with vertex set $V$,
and directed edge set $\velam$.
Each undirected edge is considered to
underlie two directed edges with opposite
orientations.
%
For an edge $e\in \velam$, the symbol $\overline{e}$
denotes the directed edge associated
with the same undirected edge as $e$ but
with opposite orientation.
The initial vertex of $e$ will be called $i(e)$ and the
terminal vertex $t(e)$.

\begin{defn} A \emph{graph of groups} is a quadruple
$\gog=(\Lambda, \{G_v\}, \{G_e\}, \{\homom_e\})$, where
$\Lambda$ is a graph, $\{G_v\}$ is a collection of groups
indexed by $V$, $\{G_e\}$ is a collection of groups
indexed by $\velam$ subject to
the condition that for all $e \in \velam$,
$G_e=G_{\overline{e}}$, and
$\{\homom_e\}$ is a collection of injective
homomorphisms
$\homom_e\colon G_e\hookrightarrow G_{t(e)}$.
\end{defn}

\begin{defn}\label{def:gog}
Let $\gog=(\Lambda, \{G_v\}, \{G_e\}, \{\homom_e\})$
be a graph of groups and let
$\tlam$ be a spanning tree of $\Lambda$.
The \emph{fundamental group} of $\gog$ at $\tlam$, denoted
$\pi_1(\gog)=\pi_1(\gog,\tlam)$,
is the group generated by the union of all of the groups $G_v$
and the set $\velnot$ of edges in $\velam$ whose underlying
undirected edge is not in $\tlam$, with three types of relations:
\begin{enumerate}
\item $\overline{e}=e^{-1}$ for all $e \in \velnot$,
\item $\homom_e(g)=\homom_{\overline{e}}(g)$ for all
$e$ in $\velam \setminus \velnot$ and $g\in G_e$, and
\item $e\homom_e(g)e^{-1}=\homom_{\overline{e}}(g)$ for all
$e \in \velnot$ and $g\in G_e$.
\end{enumerate}
\end{defn}

The fundamental group of a graph of groups can be obtained
by iterated HNN extensions (corresponding to the edges in $\velnot$)
and amalgamated free products
(corresponding to edges in $\velam \setminus \velnot$).
It is also the fundamental group of the corresponding graph of
spaces formed from the disjoint union of
Eilenberg-MacLane spaces $\{K(G_v,1)\}_{v\in V}$ by
adding tubes corresponding to $\{K(G_e,1)\times I\}_{e\in \velam}$ and
gluing the tubes using identifications corresponding to the maps
$\homom_e$ and $\homom_{\overline{e}}$.


\subsection{Relatively hyperbolic groups}\label{sub:relhyp}


$~$

\medskip

Background and details on relatively hyperbolic groups used in this
paper can be found in~\cite{osin, Farb:relhyp, Hruska:relhyp,
AntolinCiobanu:relhyp}.

For a group $G$ with a finite inverse-closed generating set $A$,
let $\vec P_{\Gamma_{A}(G)}$ denote the set of directed paths in the
associated Cayley graph. Given $p,q \in \vec P_{\Gamma_A(G)}$,
write $\vst(p)$ for the group element labeling the initial
vertex of $p$, and $\vend(p)$ for the terminal vertex.
Given $\lambda \ge 1$ and $\epsilon \ge 0$,
the path $p$ is a \emph{($\lambda,\epsilon)$-quasigeodesic}
if for every subpath $r$ of $p$ the inequality
$\ell(r) \le \lambda d_{\Gamma_A(G)}(\vst(r),\vend(r))+\epsilon$
holds, where $d_{\Gamma_A(G)}$ is the path metric distance
in $\Gamma_A(G)$.

\begin{defn}
Let $G$ be a group with a finite inverse-closed generating set $A$
and let $\{H_1,...,H_n\}$ be a collection of proper subgroups of
$G$. Let $\Gamma_A(G)$ be the Cayley graph of $G$ with respect to
$A$. For each index $j$ let $\widetilde{H_j}$ be a set in bijection
with $H_j$, and let $\HH:=\coprod_{j=1}^n (\widetilde H_j \setminus
\{\groupid\})$.
\begin{itemize}
\item The \emph{relative} Cayley graph of $G$ relative
to $\HH$, denoted $\Gamma_{A\cup \HH}(G)$ is the Cayley
graph of $G$ with generating set
$A\cup \HH$ (with the natural map from $(A \cup \HH)^* \ra G$).
\item A path $p$ in $\Gamma_{A\cup \HH}(G)$ \emph{penetrates}
a left coset $gH_{j}$ if $p$ contains an edge labeled by a letter in
$\widetilde{H_j}$ connecting two vertices in $gH_{j}$.
\item An \emph{$H_j$--component} of a path $p$ in $\Gamma_{A\cup \HH}(G)$
is a nonempty subpath $s$ of $p$ labeled by a word in $\widetilde{H_j}^*$
that is not properly contained in a longer subpath of $p$ with label
in $\widetilde{H_j}^*$.
\item A path $p \in \vec P_{\Gamma_{A\cup \HH}(G)}$ is \emph{without backtracking}
if whenever $1 \le j \le n$ and
the path $p$
is a concatenation of subpaths $p=p'srs'p''$
with two $H_j$--components $s,s'$, then the initial vertices
$\vst(s)$, $\vst(s')$ lie in different left cosets of $H_j$
(intuitively, $p$ penetrates every left coset at most once).
\end{itemize}
\end{defn}

Geometrically, the relative Cayley graph collapses each left coset
of a subgroup in $\{H_1,...,H_n\}$ to a diameter $1$ subset.

The following definition is a slight modification
of the definition originally due to Farb~\cite{Farb:relhyp},
which is shown by Osin in~\cite[Appendix]{osin} to
be equivalent to both Farb's and Osin's definitions
of relative hyperbolicity for finitely generated groups.
Many other equivalent definitions can also be found in the
literature (see for example~\cite[Section 3]{Hruska:relhyp}
for a list of many of these).

\begin{defn}\label{def:relhyp}
Let $G$ be a group with a finite generating set $A$ and let
$\{H_1,...,H_n\}$ be a collection of proper subgroups. $G$ is
\emph{hyperbolic relative to} $\{H_1,...,H_n\}$ if:
\begin{enumerate}
\item $\Gamma_{A\cup\HH}(G)$ is Gromov hyperbolic, and
\item given any $\lambda\ge 1$ there exists
a constant $B(\lambda)$ such that for any $1 \le j \le n$ and any
two $(\lambda, 0)$--quasigeodesics $p, q \in \vec P_{\Gamma_{A\cup
\HH}(G)}$ without backtracking that satisfy $\vst(p)=\vst(q)$ and
$d_{\Gamma_{A}(G)}(\vend(p),\vend(q)) \le 1$, the following hold:
\begin{enumerate}
\item If $s$ is an $H_j$--component of $p$ and
the path $q$ does not penetrate the coset $\vst(s)H_j$, then
$d_{\Gamma_{A}(G)}(\vst(s),\vend(s)) < B(\lambda)$.
\item If $s$ is an $H_j$--component of $p$ and
$s'$ is an $H_j$--component of $q$ satisfying
$\vst(s)H_j=\vst(s')H_j$, then $d_{\Gamma_{A}(G)}(\vst(s),\vst(s'))
< B(\lambda)$ and $d_{\Gamma_{A}(G)}(\vend(s),\vend(s')) <
B(\lambda)$.
\end{enumerate}
\end{enumerate}
\end{defn}


Property (1) in Definition~\ref{def:relhyp} is sometimes called
\emph{weak relative hyperbolicity} and the property (2) is called
\emph{bounded coset penetration}. The collection $\{H_1,...,H_n\}$
is called the set of \emph{peripheral} or \emph{parabolic}
subgroups. A form of bounded coset penetration for
$(\lambda,\epsilon)$--quasigeodesics is given by Osin
in~\cite[Theorem~3.23]{osin}; we record that here for
use in Section~\ref{sec:cosetautgp}.

\begin{prop}\cite[Theorem~3.23]{osin}\label{prop:bcpwithepsilon}
Let $G$ be a group with a finite generating set $A$ that is
hyperbolic relative to $\{H_1,...,H_n\}$. Given any $\lambda\ge 1$
and $\epsilon \ge 0$ there exists a constant $B=B(\lambda,\epsilon)$
such that the statement of Definition~\ref{def:relhyp}(2) holds for
any two $(\lambda, \epsilon)$--quasigeodesics $p,q$.
\end{prop}


\begin{rmk}\label{rmk:conjugateisperipheral}
We note that if a subgroup $H$ of $G=\langle A \rangle$ is
a peripheral subgroup in a relatively hyperbolic structure
for $G$ (that is, $G$ is hyperbolic relative to
$\{H_1,...,H_n\}$ and $H=H_i$ for some $i$), and
if $g$ is any element of $G$, then the conjugate
subgroup $gHg^{-1}$ is also a peripheral subgroup
in a relatively hyperbolic structure for $G$
(namely $\{gH_1g^{-1},...,gH_ng^{-1}\}$).  In particular, the isometry
$\Gamma_{A \cup \HH}(G) \ra
\Gamma_{gAg^{-1} \cup (\cup (\widetilde{gH_jg^{-1}} \setminus \{\groupid\}))}(G)$
preserves both hyperbolicity and bounded coset penetration.
\end{rmk}

When a finitely generated group $G$ is hyperbolic relative to
$\{H_1,...,H_n\}$, each of the subgroups $H_j$ is also finitely
generated~\cite[Proposition~2.29]{osin}. Since relative
hyperbolicity of the pair $(G,\{H_1,...,H_n\})$ is independent of
the finite generating set for $G$~\cite[Theorem~2.34]{osin}, for the
remainder of the paper we assume that for any relatively hyperbolic
group that $A\cap H_j$ generates $H_j$ for all $1 \le j \le n$.

\begin{defn}\cite[Construction~4.1]{AntolinCiobanu:relhyp}
Let $p$ be a path in $\Gamma_A(G)$ with label
$w \in A^*$, and write $w=w_0u_1w_1 \cdots u_nw_n$
with each $w_k \in (A \setminus (\cup_{j=1}^n (A \cap H_j)))^*$
and each $u_k \in (A \cap H_{j_k})^*$ for some index
$j_k$, such that whenever $w_i=\emptyword$ and
$x$ is the first letter of $u_{i+1}$, then $u_ix$
does not lie in $(A \cap H_j)^*$ for any $j$.
The path $p$ has
\emph{no parabolic shortenings} if for each $k$ the subpath
labeled by the subword
$u_k$ is a geodesic in the subgraph
$\Gamma_{A\cap H_{j_k}}(H_{j_k})$.
If the path $p$ has no parabolic shortenings, let
$\hat{p}$ be the path in $\Gamma_{A\cup\HH}(G)$
with label $w_0h_1w_1 \cdots h_nw_n$ where
for each $k$ the symbol $h_k$ denotes the letter
in $\HH$ representing the element $u_k$ of $H_{j_k}$;
the path $\hat{p}$
is the path \emph{derived from} $p$.
\end{defn}

Antolin and Ciobanu studied geodesics and language theoretic
properties of relatively hyperbolic groups
in~\cite{AntolinCiobanu:relhyp}. For example, they show that every finitely
generated relatively hyperbolic group has a finite generating set,
$A$, such that each $\Gamma_{A\cap H_j}(H_j)$ isometrically embeds
in $\Gamma_A(G)$~\cite[Lemma~5.3]{AntolinCiobanu:relhyp} (in fact every
finite generating set can be extended to one with this property).
We apply the following results from their paper in Section~\ref{sec:cosetautgp}.

\begin{defn}\label{def:nice}
Let $G$ be a finitely generated group hyperbolic relative to
$\{H_1,...,H_n\}$ and suppose that $\lambda \ge 1$ and $\epsilon \ge
0$. A finite inverse-closed generating set $A$ for $G$ is called
\emph{$(\lambda,\epsilon)$--\nice} if
\begin{enumerate}
\item every path derived from a  geodesic in $\Gamma_A(G)$
is a $(\lambda, \epsilon)$--quasigeodesic in $\Gamma_{A\cup\HH}(G)$
without backtracking,
\item $H_j=\langle A \cap H_j \rangle$ for all $j$, and
\item for every total ordering on $A$
satisfying the property that $H_j$ is shortlex biautomatic
on $A\cap H_j$ (with the restriction of the ordering
from $A$) for all $j$, the group $G$ is shortlex biautomatic over $A$
with respect to that ordering.
\end{enumerate}
\end{defn}

\begin{thm}\cite[Lemma~5.3, Theorem~7.7]{AntolinCiobanu:relhyp}\label{thm:quasigeodesic}
Let $G$ be a group with finite generating set $A'$ that is
hyperbolic relative to $\{H_1,...,H_n\}$. Then there are constants
$\lambda\ge 1$ and $\epsilon\ge 0$ and a finite subset $\HH'$ of
$\HH$ such that every finite generating set $A$ of $G$ satisfying
$A' \cup \HH' \subseteq A \subseteq A' \cup \HH$ is a
$(\lambda,\epsilon)$--{\nice} generating set.
\end{thm}

\begin{rmk}\label{rmk:toralhypisnice}
In~\cite[Theorem~4.3.1]{wordprocessing}, Holt shows that
every finitely generated abelian group is shortlex
automatic over every finite generating set with respect to
every ordering of that set; moreover,
the structure is also biautomatic.  Combining this with
Theorem~\ref{thm:quasigeodesic} implies
that any finitely generated
group hyperbolic relative to abelian subgroups
is shortlex biautomatic (on a $(\lambda,\epsilon)$--{\nice}
generating set), and hence autostackable.
\end{rmk}


\subsection{3-manifolds}\label{sub:3mfld}


$~$

\medskip

We review some important facts about 3-manifolds that will be used
later in the paper. For background, an interested reader can consult
\cite{Neumann:3mflds, Scott:3mflds, thurston}. Let $M=M^3$ be a
connected, compact, orientable, three-dimensional manifold with
\emph{incompressible toral boundary}; that is, the boundary of $M$
consists of a finite number of incompressible ({i.e.},
$\pi_1$--injective) tori.

\begin{defn}
A 3-manifold $M=M^3$ is called \emph{prime} if whenever $M$
is a connected sum $M\cong M_1\#M_2$, then one of $M_1$ or $M_2$
is homeomorphic to $S^3$.
\end{defn}

Decomposing the compact 3-manifold $M$  along a disjoint collection
of $S^2$'s via the connected sum operation,
there is a decomposition $M=M_1\#\cdots\#M_k$, where each of the
$M_i$ are prime, which is unique up to reordering. This gives a
decomposition of $\pi_1(M)$ as a free product of the fundamental
groups of its prime factors.

If $M$ is prime then, using Thurston's Geometrization Conjecture (proved
by Perelman; see, e.g.,~\cite{morgantian}), either $M$ admits a
geometric structure based on one of $S^3, S^2\times \R,
\mathbb{E}^3, \mathbb{H}^2\times \R, \widetilde{PSL_2}, Nil, Sol$ or
$\mathbb{H}^3$, in which case $M$ is called \emph{geometric}, or
else $M$ contains an incompressible torus.

In the nongeometric case, geometrization says that $M$ can be split
along a collection $\{T_i\}$ of non-isotopic, incompressible,
two-sided tori in $M$ in such a way that every connected component of
$M\setminus\cup N(T_i)$ (where each $N(T_i)$ is an open regular
neighborhood of $T_i$), known as a \emph{piece}, has interior
admitting a geometric structure with finite volume; this is commonly
referred to as a JSJ decomposition.
Moreover each piece in $M\setminus\cup N(T_i)$ is either Seifert
fibered or atoroidal, and  again by geometrization,
the atoroidal pieces have interior that is
hyperbolic. In this nongeometric case the fundamental
group $\pi_1(M)$ decomposes as
the fundamental group of a graph of groups
whose vertex groups are the fundamental groups of the pieces and whose
edge groups are all $\Z^2$ subgroups corresponding to
fundamental groups of the tori $T_i$ in the decomposition.

Now suppose that $N$ is a piece from the JSJ decomposition in
the nongeometric case, or else that $N=M$ in the geometric case.  
Then $N$ is also a compact 3-manifold with incompressible toral boundary.
A collection of $\Z^2$ subgroups arising from this boundary, one for
each free homotopy class of boundary component of $N$, or
equivalently one for each conjugacy class of $\Z^2$ subgroup, is a
collection of \emph{peripheral subgroups} of $\pi_1(N)$.

If $N$ is a Seifert fibered $3$-manifold with boundary, then
$N$ is a circle bundle over a two-dimensional orbifold
with boundary. Consequently, $\pi_1(N)$ is an extension
of the orbifold fundamental group of that
$2$-dimensional orbifold by $\Z$ (see~\cite[Lemma~3.2]{Scott:3mflds}
for more details).
On the other hand, if the interior of $N$ is a finite volume
hyperbolic $3$-manifold, then $\pi_1(N)$ is
hyperbolic relative to a collection of peripheral
($\Z^2$) subgroups (see~\cite[Theorem~5.1]{Farb:relhyp}).


\section{Autostackability for graphs of groups}\label{sec:autstkgog}


In this section we prove Theorem~\ref{thm:GoGautostack}, the closure
of autostackability under the construction of fundamental groups of
graphs of groups, in the case that the vertex groups are
autostackable {\rsp} their respective edge groups.

We begin by noting that a small extension of the
proof that autostackability is invariant
under increasing the generating
set~\cite[Proposition~4.3]{HermillerMartinez:HNN}
yields the following Lemma, which will be useful in our
closure proof.

\begin{lem}\label{lem:extendgen}
Let $G$ be autostackable over a generating
set $A$ {\rsp} a subgroup $H$, and let $A' \supseteq A$
be another finite inverse-closed generating set for $G$.
Then $G$ is also autostackable over $A'$ {\rsp} $H$.
\end{lem}

\begin{proof}
The set of normal forms over $A'$
is taken to be the same as the normal form set over $A$,
and the flow function on edges with labels in $A$
is also unchanged.
The flow function maps edges labeled by any letter
$a \in A' \setminus A$ to paths labeled by $\nf{a}$.
\end{proof}

\begin{lem}\label{lemma:respnf}
Let $G$ be autostackable over a generating set $A$ {\rsp} a subgroup
$H$ with generating set $B \subseteq A$. Suppose that the
set of normal forms is
$\Nf_{G}=\Nf_{H}\Nf_{\Tr}$, where
$\Nf_{H}$ is a set of normal forms for $H$ over $B$ and $\Nf_{\Tr}$
is a set of normal forms over $A$ for a right transversal of $H$ in
$G$. Suppose further that $\Nf_G$ is regular and prefix-closed.  Then
$$
\Nf_{H} = \Nf_{G} \cap B^*
\hspace{.2in} \text{ and } \hspace{.2in}
\Nf_{\Tr} = \{\emptyword\} \cup
    \big(\Nf_{G} \cap [(A \setminus B) \cdot A^*]\big),
$$
and both $\Nf_{H}$ and $\Nf_{\Tr}$
are also regular prefix-closed languages.
\end{lem}

\begin{proof}
Note that since $\Nf_G$ is prefix-closed,
the empty word $\emptyword$ is an element of $\Nf_G$,
and so we also have $\emptyword \in \Nf_H$
and $\emptyword \in \Nf_{\Tr}$.

Let $w$ be any word in $\Nf_{\Tr}$ and write
$w=yz$ where $y$ is the maximal prefix of $w$
representing an element of $H$.
Let $y'\in \Nf_{H} \subset B^{*}$ be
the normal form for the inverse of $y$.
Then the word $y'w=y'yz$ is in $\Nf_{G}$,
and so by prefix-closure of this set, $y'y \in \Nf_{G}$
as well.  Since $y'y =_{G} \groupid$, and
normal forms are unique, then $y=y'=\emptyword$.
Hence no word in $\Nf_{\Tr}$ has a nonempty
prefix representing an element in $H$.
The rest of the result follows from the closure
properties of regular languages.
\end{proof}

Next we establish some notation used throughout the
rest of this section.

Let $\gog=(\Lambda, \{G_v\}, \{G_e\}, \{\homom_e\})$ be a graph of
groups on a finite connected graph $\Lambda$ with
vertex set $V$,
basepoint $v_0 \in V$,
directed edge set $\velam$, and
spanning tree $\tlam$.
Let  $\velnot \subseteq \velam$ denote the subset
of $\velam$ of directed edges whose underlying
undirected edges do not lie in $\tlam$.
For each $v \in V$ let $A_v$ be a finite inverse-closed
generating set for $G_{v}$.
Let $\widetilde A:= (\cup_{v \in V} A_v)  \cup \velam$
and $A := (\cup_{v \in V} A_v) \cup \velnot$.

The fundamental group $\pi_1(\gog)$ of this graph of groups is a
quotient of the free product $(*_{v \in V} G_v) * F(\velnot)$, where
$F(\velnot)$ is the free group on the set $\velnot$, satisfying
$\overline{e}=e^{-1}$ for each $e \in \velnot$; thus, the set $A$ is
a finite inverse-closed generating set for $\pi_1(\gog)$. With each
of the elements of $\velam \setminus \velnot$ as representatives of
the identity of this group, the set $\widetilde A$ also is a finite
inverse-closed generating set for $\pi_1(\gog)$.

Define two functions $\vst,\vend: \widetilde A^* \ra V$
as follows.
Let $\vst(\emptyword)=\vend(\emptyword):=v_0$.
For each letter $a \in \velam$, define $\vst(a)$ to be
the initial vertex of $a$ and $\vend(a)$ to be the terminal
vertex of $a$, and
for each vertex $u \in V$
and letter $a \in A_u$, define $\vst(a)=\vend(a):=u$.
Finally, for an arbitrary nonempty word $w \in A^*$, define
$\vst(w)$ to be $\vst(a)$ where $a$ is the first letter of $w$,
and define $\vend(w)$ to be $\vend(b)$ where $b$ is the last
letter of $w$.

Let $\edgeonly:\widetilde A^* \ra \velam^*$ be the monoid
homomorphism determined by $\edgeonly(a):=a$ for all $a \in \velam$
and $\edgeonly(a):=\emptyword$ for all $e \in \cup_{v \in V} A_v$.
Given any two vertices $u,v \in V$, let $\pt(u,v)$ denote
the word in $(\velam \setminus \velnot)^*$ that is the unique
path without backtracking in the tree $\tlam$ from $u$ to $v$.
Note that if $u=v$ then $\pt(u,v)=\emptyword$.

We construct an ``inflation'' function
$\infl: A^* \ra \widetilde A^*$ by defining $\infl$ on any
word $w=a_1 \cdots a_k$, with each $a_i \in A$, as
$$
\infl(w)  :=  \pt(v_0,\vst(a_1))a_1\pt(\vend(a_1),\vst(a_2))a_2
\cdots \pt(\vend(a_{n-1}),\vst(a_n))a_n.
$$
Note that for any word $w$ over $A$, the word
$\edgeonly(\infl(w))$ is a directed edge path in $\Lambda$;
moreover, it is the shortest directed
edge path in $\Lambda$ starting at
the basepoint $v_0$ and ending at the vertex
$\vend(w)$ that traverses the vertices
$\vend(a_i)$ of the letters $a_i$ that lie in $\cup_{v \in V} A_v$
and the edges corresponding to the letters $a_i$
lying in $\velnot$, in the order that they
appear in $w$, and otherwise remains in the tree $\tlam$.

Define a ``deflation'' (monoid) homomorphism
$\defl:\widetilde A^* \ra A^*$
by $\defl(a):=a$ for all $a \in A$
and $\defl(e):=\emptyword$ for all $e \in \velam \setminus \velnot$.
Note that for any word $w \in A^*$ we have
$\defl(\infl(w))=w$.

Finally, define a ``pruning'' function
$\trim:\widetilde A^* \ra \widetilde A^*$ by
$\trim(w) :=$ the maximal prefix of $w$ whose
last letter lies in $A$.
Since $e$ represents the
identity of $\pi_1(\gog)$ for all 
$e \in \velam \setminus \velnot$,
all three maps $\infl$, $\defl$, and $\trim$ preserve
the group element represented by the word. 

Each element $g$ of the fundamental group $\pi_1(\gog)$
is represented by a loop at $v_0$
in the graph $\Lambda$ together with
vertex group elements at each vertex group traversed
by the loop; that is, 
$g =_{\pi_1(\gog)} g_0 e_1 g_1 \cdots e_k g_k$ where
$e_1 \cdots e_k$ is a loop in $\Lambda$ based
at $v_0$, $g_0 \in G_{v_0}$, and 
$g_i \in \vend(e_i)$ for each $i$.
This format for elements is realized
by the inflation function;
given any word over $A$ representing an element of 
$\pi_1(\gog)$, the inflation function 
inserts letters that are
edges in the tree in order to produce a word 
representing a path starting at the basepoint
in the graph $\Lambda$, interspersed
with letters from vertex groups at vertices along
that path; however the inflation function
does not add a final path through the tree 
to make the path end at the basepoint.
The deflation map reverses this
process, removing all letters corresponding
to edges in the tree.

With this notation, we use a construction
of normal forms for fundamental groups given
by Higgins in~\cite{HigginsNF} to
obtain a
set of normal forms for elements of $\pi_1(\gog)$.

\begin{prop}\label{prop:gognf}
Let $\gog=(\Lambda, \{G_v\}, \{G_e\}, \{\homom_e\})$ be a graph of
groups on a finite connected graph $\Lambda$ with
vertex set $V$,
basepoint $v_0 \in V$,
directed edge set $\velam$,
spanning tree $\tlam$,
and subset $\velnot \subseteq \velam$ consisting of edges
not lying in $\tlam$.
For each $v \in V$ let $A_v$ be a finite inverse-closed
generating set for $G_{v}$.
Let $\widetilde A:= (\cup_{v \in V} A_v)  \cup \velam$
and $A := (\cup_{v \in V} A_v) \cup \velnot$, and
let $\defl\colon \widetilde A^* \ra A^*$ be the deflation map.
Suppose that $\Nf_0$ is a regular prefix-closed
set of normal forms for $G_{v_0}$ over $A_{v_0}$, and for each $e
\in \velam$ suppose that $\Nf_{\Tr,{e}}$ is a regular prefix-closed
set of normal forms for a right transversal over $A_{t(e)}$ of
$\homom_e(G_e)$ in $G_{t(e)}$. Define
\begin{align*}
\widetilde{\Nf}:=\{w_0e_1t_1e_2t_2\cdots e_kt_k \mid & k \ge 0,
  \vst(e_1)=v_0, w_0\in \Nf_0, \forall~i,~e_i \in \velam, \\
& \vend(e_i)=\vst(e_{i+1}),
   t_i\in \Nf_{\Tr,{e_i}}, \text{ and }
   t_i \neq \emptyword \text{ if } e_{i+1}=\overline{e_i}\\
& \text{and either } t_k \neq \emptyword,~e_k \in \velnot, \text{ or } k=0\},
\end{align*}
and let
$\Nf:=\defl\left(\widetilde{\Nf}\right)$.
Then $\widetilde \Nf$ is a regular set of
normal forms over $\widetilde A$,
and $\Nf$ is a regular prefix-closed
set of normal forms over $A$,
for the fundamental group $\pi_1(\gog)$.
\end{prop}

\begin{proof}
In~\cite[Corollary~1]{HigginsNF}, Higgins proved that the set
\begin{align*}
\widehat{\Nf}:=\{w_0e_1t_1e_2t_2\cdots e_\ell t_\ell \mid & \ell \ge 0,
  e_i \in \velam, \vst(e_1)=\vend(e_\ell)=v_0, \vend(e_i)=\vst(e_{i+1}),\\
& w_0\in \Nf_0, t_i\in \Nf_{\Tr,{e_i}},
 t_i \neq \emptyword \text{ if } e_{i+1}=\overline{e_i}\}
\end{align*}
is a set of normal forms
for the fundamental group $\pi_1(\gog)$.
Note that for each word $w \in \widetilde\Nf$, the
last letter of $w$ cannot lie in $\velam \setminus \velnot$.
Hence the concatenated word $w\pt(\vend(w),v_0)$ lies in $\widehat\Nf$,
and moreover, for each word $x \in \widehat\Nf$, the pruned
word $\trim(x)$ lies in $\widetilde\Nf$.
Thus the maps $\widetilde\Nf \ra \widehat\Nf$,
defined by $w \mapsto w\pt(\vend(w),v_0)$,
and $\widehat\Nf \ra \widetilde\Nf$,
defined by $x \mapsto \trim(x)$, are inverses.
These maps preserve the group element being represented,
and so $\widetilde\Nf$ is also a set of
normal forms over $\widetilde A$ for $\pi_1(\gog)$.
Thus to prove that $\Nf$ is also a set of normal
forms for $\pi_1(\gog)$, it suffices to
prove that the restriction of the deflation map
$\defl$ to $\widetilde{\Nf}$ is a bijection.

By definition, $\Nf=\defl(\widetilde \Nf)$, and thus $\defl\colon
\widetilde \Nf\ra \Nf$ is surjective. Thus, we need only show that
it is injective.
Suppose that $w, w'\in\widetilde \Nf$ and $\defl(w)=\defl(w')$. As
deflation only eliminates edges in $\tlam$, we have that
$w=_{\pi_1(\gog)}\defl(w)$ and $w'=_{\pi_1(\gog)}\defl(w')$. Now,
this implies that $w=_{\pi_1(\gog)}w'$; however, since $\widetilde
\Nf$ is a set of normal forms for $\pi_1(\gog)$, we must have that
$w=w'$. Thus, $\defl$ is injective when restricted to $\widetilde
\Nf$.  We also note that the restriction of the inflation map
to $\Nf$ gives
$\infl:\Nf \ra \widetilde{\Nf}$, which is the inverse of $\defl$.
Thus $\Nf$ is also a set of normal forms for
$\pi_1(\gog)$.

Since $\defl\colon \widetilde A^* \ra A^*$ is a monoid homomorphism
and the image of a regular language is regular
(Theorem~\ref{thm:regclosure}), in order to show that
$\Nf$ is regular it suffices to show
that $\widetilde{\Nf}$ is regular.
For all $w\in\widetilde{\Nf}$,
the word $w$ is a concatenation $w=w_0w'$, where $w_0\in\Nf_0$
and the suffix $w'$ alternates between edges and
transversal normal forms from
the transversals corresponding to the edges; that is,
$w \in \Nf_0\cdot(\cup_{e\in \velam}~e\Nf_{\Tr,e})^*$.
On the other hand, a word
$w \in \Nf_0\cdot(\cup_{e\in \velam}~e\Nf_{\Tr,e})^*$
will fail to be in the language $\widetilde \Nf$
if and only if either $w$ contains an edge and
its reverse consecutively, the terminal vertex of one edge is not
the initial vertex of the next, the first edge in $w'$
does not have initial vertex $v_0$, or the
last letter of $w$ lies in $\velam \setminus \velnot$.
That is,
\begin{align*}
\widetilde{\Nf}=\Nf_0\cdot\left(\cup_{e\in
\velam}e\Nf_{\Tr,e}\right)^* \setminus & \Big[ \left(\cup_{e\in
\velam} \widetilde{A}^*e\overline{e}\widetilde{A}^*\right)\cup
\left(\cup_{\vend(e)\neq \vst(e')}\widetilde{A}^*e\Nf_{\Tr,e}e'\widetilde{A}^*\right) \\
&\cup\left(\cup_{\vst(e)\neq v_0}\Nf_0e\widetilde{A}^*\right) \cup
\left(\widetilde{A}^*(\velam \setminus \velnot\right) \Big].
\end{align*}
Hence $\widetilde \Nf$ is formed from regular languages by concatenation,
union, Kleene star, and complementation. Thus, by Theorem~\ref{thm:regclosure},
$\widetilde \Nf$ is regular, and so $\Nf$ is also regular.

Finally, suppose that $w'$ is a
prefix of a word $w \in \Nf$.
Then $w=w'w''$ for some $w'' \in A^*$, and so
the ``lift'' $\infl(w)$
of $w$ to $\widetilde{\Nf}$ satisfies
$\infl(w)=\infl(w')x''$ for some word $x'' \in \widetilde A^*$.
Now prefix-closure of the sets $\Nf_0$ and of $\Nf_{\Tr,e}$
for all $e\in \velam$ implies that the
word $\infl(w')$ lies in the set
$$
\Nf_0\cdot\left(\cup_{e\in \velam}e\Nf_{\Tr,e}\right)^*
\setminus
\left[
\left(\cup_{e\in \velam} \widetilde{A}^*e\overline{e}\widetilde{A}^*\right)\cup
\left(\cup_{\vend(e)\neq \vst(e')}\widetilde{A}^*e\Nf_{\Tr,e}e'\widetilde{A}^*\right)
\cup\left(\cup_{\vst(e)\neq v_0}\Nf_0e\widetilde{A}^*\right)
\right].
$$
Moreover, the last letter of $\infl(w')$ is the last
letter of $w'$, since the inflation procedure
does not insert any letters at the end of a word over $A$.
Hence $\infl(w')$ lies in $\widetilde\Nf$, and so
$w'=\defl(\infl(w'))$ lies in $\Nf$.
Therefore the set $\Nf$ is also prefix-closed.
\end{proof}

In the following lemma we illustrate the effect
on normal forms of multiplication by a word in a
vertex group.
This information will be integral to our proof in 
Theorem~\ref{thm:GoGautostack}
that the function $\Phi$ satisfies the property (F3), showing
that iteration of $\Phi$ on any path eventually stabilizes
at a path in the tree.

\begin{lem}\label{lem:computenf}
Let $\gog$, $\Lambda$, $v$, $v_0$, $\velam$, $\tlam$, $\velnot$,
$A_v$, $G_v$, $\widetilde A$, $\Nf_0$, $\Nf_{\Tr,e}$, and $\widetilde \Nf$
be as in the statement of Proposition~\ref{prop:gognf}.
Let
$\tilde y=w_0e_1t_1e_2t_2\cdots e_kt_k$ be an element
of $\widetilde\Nf$ with $k \ge 0$,
$\vst(e_1)=v_0$, $w_0\in \Nf_0$, and each $e_i \in \velam$,
$\vend(e_i)=\vst(e_{i+1})$, and
$t_i\in \Nf_{\Tr,{e_i}}$.
Moreover, let $w \in A_u^*$ for some $u \in V$
and write $\pt(\vend(\tilde{y}),\vst(w))=f_1 \cdots f_m$
with $m \ge 0$ and each $f_j \in \velam \setminus \velnot$.
Then the normal form of $\tilde{y}w$ in $\widetilde\Nf$ is
$$
\mathsf{nf}_{\widetilde\Nf}(\tilde{y}w)=\trim(w_0'
    e_1x_1e_2 \cdots e_kx_k
    f_1 w_1 f_2 \cdots f_m w_m)
$$
for some words $w_0' \in \Nf_0$,
$x_i \in \Nf_{\Tr,e_i}$, and $w_j \in \Nf_{\Tr,f_j}$.
\end{lem}

\begin{proof}
Note that since $\tilde{y} \in \widetilde\Nf$,
then $\tilde{y}$ is freely reduced
and $\tilde{y}$ does not end with a letter in $\velam \setminus \velnot$.
Hence the word $\tilde{y}\pt(\vend(\tilde{y}),\vst(w))$ is also freely reduced.

We can uniquely factor the element of $G_{\vend(f_{m})}$ represented
by $w$ as $w=_{G_{\vend(f_{m})}} g_{m}w_{m}$ where $g_{m} \in
\homom_{f_{m}}(G_{f_{m}})$ and $w_{m} \in \Nf_{\Tr,f_{m}}$. Let
$\widetilde g_{m}$ denote a representative of $g_{m}$ in
$G_{\vst(f_{m})}$; that is, $\widetilde g_{m} =
\homom_{\overline{f_m}}(\homom_{f_m}^{-1}(g_{m}))$. Repeating this
process, for each $1 \le i \le m-1$, write the element $\widetilde
g_{i+1}$ in $G_{\vend(f_{i})}$ as $\widetilde
g_{i+1}=_{G_{\vend(f_{i})}} g_{i}w_{i}$ where $g_{i} \in
\homom_{f_{i}}(G_{f_{i}})$ and $w_{i} \in \Nf_{\Tr,f_{i}}$, and let
$\widetilde g_{i}$ denote a representative of $g_{i}$ in
$G_{\vst(f_{i})}$. Then $\tilde{y}w =_{\pi_1(\gog)}
\tilde{y}\pt(\vend(y),\vst(w))w =_{\pi_1(\gog)} \tilde{y} \widetilde g_{1}f_1
w_{1} f_2 \cdots f_{m} w_{m}$.

Continuing this process further, let $\widetilde g_{k+1}':=\widetilde
g_{1}$, and for all $1 \le i \le k$, write the element $t_{i}
\widetilde g_{i+1}'$ in $G_{\vend(f_{i})}$ as $t_{i} \widetilde
g_{i+1}'=_{G_{\vend(f_{i})}} g_{i}'x_{i}$ where $g_{i}' \in
\homom_{f_{i}}(G_{f_{i}})$ and $x_{i} \in \Nf_{\Tr,f_{i}}$, and let
$\widetilde g_{i}'$ denote a representative of $g_{i}'$ in
$G_{\vst(f_{i})}$.  Also let $w_0'$ be the normal form in $\Nf_0$ of
the element represented by $w_0\widetilde g_{1}'$. Then
$$\tilde{y}w=_{\pi_1(\gog)} yw
 =_{\pi_1(\gog)} w_0'e_1x_{1}e_2 \cdots e_kx_{k}
 f_1 w_{1} f_2 \cdots f_{m} w_{m}.
$$

If there is an index $i$ such that $x_{i}=\emptyword$, then $t_{i}
=_{\pi_1(\gog)} g_{i}'(\widetilde g_{i+1}')^{-1}$ and since
$g_i',\widetilde g_{i+1}'$ are in $\homom_{f_{i}}(G_{f_{i}})$, then
$t_i$ is in $\homom_{f_{i}}(G_{f_{i}})$ as well. Since $t_i$ is
an element of $\Nf_{\Tr,f_{i}}$, then $t_i=\emptyword$ in this case.
Note that if $x_k=\emptyword$, then $t_k=\emptyword$, and since
$\tilde{y}$ is pruned, this implies that $e_k\in\velnot$. Thus, the
word $\tilde{y}\pt(\vend(\tilde{y}),\vst(w))$ does not contain any
consecutive inverse pair of letters, and the process in this proof
cannot create any subword that is not freely reduced. Thus the word
produced by this process satisfies all of the properties of words in
$\widetilde\Nf$, except for the possibility that it has a suffix of
letters in $(\velam \setminus \velnot)^*$. Therefore
$\mathsf{nf}_{\widetilde\Nf}(\tilde{y}w)=\trim(w_0'
    e_1x_{1}e_2 \cdots e_kx_{k}
    f_1 w_{1} f_2 \cdots f_{m} w_{m}).$
\end{proof}

We are now
ready to show autostackability
of fundamental groups of graphs of autostackable groups
in which each vertex
group has an autostackable structure {\rsp} each
incident edge group.  En route we also
prove this closure property for stackable groups.

\begin{thm}\label{thm:GoGautostack}
Let $\gog$ be a graph of groups
over a finite connected graph $\Lambda$ with at least
one edge.
If for each
directed edge $e$ of $\Lambda$
the vertex group $G_v$ corresponding
to the terminal vertex $v=t(e)$ of $e$
is autostackable [respectively, stackable] {\rsp}
the associated injective homomorphic image
of the edge group $G_e$, then the
fundamental group
$\pi_1(\gog)$ is autostackable [respectively, stackable].
\end{thm}

\begin{proof}
Let $\gog=(\Lambda, \{G_v\}, \{G_e\}, \{\homom_e\})$ be a graph of
groups with vertex set $V$,
basepoint $v_0 \in V$,
directed edge set $\velam$,
spanning tree $\tlam$, and subset $\velnot \subseteq \velam$ of edges
not lying in $\tlam$,
of the finite connected graph $\Lambda$.

Since $\Lambda$ is connected and has at
least one edge, every vertex of $\Lambda$
is the terminus of an edge in $\Lambda$,
including $v_0$, and so every vertex group
is autostackable.
Suppose that $G_{v_0}$ has an autostackable
structure over an alphabet $A_{0}$ and normal
form set $\Nf_0$.  Also for each
$e\in \velam$, and $v\in V$ with $t(e)=v$,
suppose that $G_v$ is autostackable {\rsp} $\homom_e(G_e)$
over a finite inverse-closed generating set
$A_{v,e}$, with normal form set of the form
$\Nf_{G_{v},e}=\Nf_{\homom_e(G_e)}\Nf_{\Tr,e}$
where $\Nf_{\homom_e(G_e)}$ is a set of normal
forms for $\homom_e(G_e)$ over a finite
inverse-closed generating set $B_e$ of $\homom_e(G_e)$
contained
in $A_{v,e}$, and $\Nf_{\Tr,e}$ is a set
of normal forms over $A_{v,e}$ for a right transversal
of $\homom_e(G_e)$ in $G_v$.

For each $v \in V \setminus \{v_0\}$, let
$A_v := \cup_{e \in \velam, t(e)=v} A_{v,e}$, and let
$A_{v_0} := A_0 \cup (\cup_{e \in \velam, t(e)=v_0} A_{v_0,e})$.
By Lemma~\ref{lem:extendgen}, the group
$G_{v_0}$ is also autostackable over $A_{v_0}$
with normal form set $\Nf_0$, and for
each vertex $v$ and edge $e$ with $t(e)=v$ the group
$G_{v}$ is autostackable {\rsp} $\homom_e(G_e)$
over $A_v$ with the same normal form set
$\Nf_{G_{v},e}$.  Let $\Phi_0$ be the associated
flow function for $G_0$, with bound $K_0$,
and let $\Phi_e$ denote the flow function for
the autostackable structure on $G_v$
{\rsp} $\homom_e(G_e)$, with bound $K_e$.

Let $A := (\cup_{v \in V} A_{v}) \bigcup \velnot$.
Let $\Gamma$ be the Cayley graph of $\pi_1(\gog)$ over $A$ and as
usual let $\vec E, \vec P$ be the sets of directed edges and paths
in $\Gamma$.
Having satisfied all of the hypotheses of Proposition~\ref{prop:gognf},
let $\Nf$ be the regular prefix-closed set of normal
forms for $\pi_1(\gog)$ over $A$ from that Proposition. 

For every $g \in \pi_1(\gog)$, let $\nf{g} \in \Nf$
denote the normal form of $g$.  Let $T$ be the spanning
tree in $\Gamma$ determined by the normal form set $\Nf$.

\medskip

\noindent\textbf{The flow function and stackability:}

Next we define a (stacking) function $\phi:\Nf \times A \ra A^*$
and show that the associated function
$\Phi:\vec E \ra \vec P$ defined by $\Phi(e_{g,a}) :=$
the path in $\Gamma$ starting at $g$ labeled by $\phi(\nf{g},a)$
satisfies the properties of a bounded flow function.

Let $\phi_0:\Nf_0 \times A_{v_0} \ra A_{v_0}^*$ be the stacking
function associated to $\Phi_0$, and let
$\phi_e:\Nf_{G_{t(e),e}} \times A_{t(e)} \ra A_{t(e)}^*$
be the stacking functions associated to the
flow functions $\Phi_e$ for each
$e \in \velam$, respectively.

Before giving the full details of the 
definition of the stacking function $\phi$, 
we include an informal description.
Given a normal form word $y \in \Nf$ and a letter $a \in A$,
there are three possible options.
The first is that $a \in \velnot$, in which case 
we define $\phi(y,a)=a$ so that the flow
function fixes the edge $e_{y,a}$.
Otherwise $a \in A_v$ for some vertex $v$ of $\Lambda$.
Let $p$ be the path in $\Lambda$ associated to
the inflation of the 
word $ya$.
The second option is that the path $p$
ends with an edge $e$ such that  
$a$ is a generator of the edge group $h(G_e)$
and $y$ does not end with a letter in $A_v$;
then we use relation (2) or (3) of Definition~\ref{def:gog}
(depending on whether or not $e$ is in $\tlam$)
to define $\phi(y,a)$ so that 
the path in $\Lambda$ associated to the
inflation of the normal form for $y\phi(y,a)$
is shorter.  Finally, in the third
option (that is, in any other case)
 we apply the stacking function
$\phi_e$ (or $\phi_0$ if $p$ is the empty path).
To make this more precise, we again define some notation.

For any word $w \in A^*$ and letter $a \in \cup_{v \in V}A_v$,
the word $\edgeonly(\infl(wa))$
is a directed edge path in $\Lambda$ from $v_0$
via $\vend(w)$, to $\vst(a)$, and in particular this path
satisfies
$\edgeonly(\infl(wa))=\edgeonly(\infl(w))\pt(\vend(w),\vst(a))$.
Let $\epair(w,a):=0$ if $\edgeonly(\infl(wa))=\emptyword$
(equivalently, if $w \in A_{v_0}^*$ and $a \in A_{v_0}$),
and let $\epair(w,a)$ be the last
letter of $\edgeonly(\infl(wa))$ otherwise.
That is, $\epair(w,a)$
is the last edge encountered on the path $\edgeonly(\infl(wa))$.
For each word $w$ in $A^*$ (or $\widetilde A^*$) and vertex
$u \in V$, let $\suf_u(w)$ denote the maximal suffix
of $w$ contained in $A_u^*$.

For any edge $e \in \velam$
and letter $a \in B_e$, let $\widehat{a}_e$ denote
the normal form over $B_{\overline{e}}$
in the autostackable structure of
${G_{\vst(e)}}$ {\rsp} $\homom_{\overline{e}}(G_{\overline{e}})$
of the element ${\homom_{\overline{e}}(\homom_{e}^{-1}(a))}$;
that is, if $e \in \velam \setminus \velnot$ then
$a=_{G} \widehat{a}_{e} $ and
if $e \in \velnot$ then $a=_{G} e^{-1}\widehat{a}_{e} e$.


Let $y \in \Nf$ and let $a \in A$.
The function $\phi:\Nf \times A \ra A^*$ evaluated at $(y,a)$
is given by
$$
\phi(y,a):=
\begin{cases}
a             & \text{ if } a \in \velnot \\
\defl(\epair(y,a)^{-1}~\widehat{a}_{\epair(y,a)}~\epair(y,a))
              & \text{ if } \epair(y,a) \in \velam,
                a \in B_{\epair(y,a)}, \\
              & \hspace{.1in} \text{ and } \suf_{\vst(a)}(y)=\emptyword \\
\phi_{\epair(y,a)}(\suf_{\vst(a)}(y),a) & \text{ otherwise.}
\end{cases}
$$
Let $\Phi:\vec E \ra \vec P$ be defined by $\Phi(e_{g,a}) :=$
the path in $\Gamma$ starting at $g$ labeled by $\phi(\nf{g},a)$,
for all $g \in G$ and $a \in A$.

\noindent{\bf Property (F1):} It follows immediately from this
definition that for each directed edge $e \in \vec E$ in the Cayley
graph $\Gamma$, the path $\Phi(e)$ has the same initial and terminal
vertices as $e$. Also the length of the path $\Phi(e)$ is at most
$2+\max(\{\ell(\widehat{a}_e) \mid e \in \velam, a \in B_e\} \cup
   \{K_f \mid f \in \{0\} \cup \velam\})$.
This is a maximum over a finite set, hence (F1) holds for $\Phi$.

\noindent{\bf Property (F2):} To check that $\Phi$ fixes directed edges lying
in the spanning tree $T$ in $\Gamma$, suppose that
$e_{g,a}$ is a directed edge whose underlying undirected
edge lies in $T$, and let $y:=\nf{g}$.
Then either $\nf{ga}=ya$, or
else $y$ ends with the letter $a^{-1}$.

If $a \in \velnot$, then $\Phi(e_{g,a})=e_{g,a}$ from
the definition above, as required.

On the other hand, suppose that $a \in A_u$ for some $u \in \velam$.
Now let $\tilde y = \infl(y)$, so that $y=\defl(\tilde y)$. We can
write $\tilde y=w_0e_1t_1e_2t_2\cdots e_kt_k \in \widetilde\Nf$ such
that $e_1 \cdots e_k$ is an edge path in $\Lambda$ from $v_0$ to
$\vend(y)$,
$w_0 \in \Nf_0$, each $t_i \in \Nf_{\Tr,e_i}$,
no inverse pair $e\overline{e}$ appears as a subword,
and the last letter of $\tilde y$ lies in $A$ (or $\tilde y=\emptyword$).
Now $\infl(ya)=
\infl(y) \pt(\vend(y),\vst(a))a$
and $\edgeonly(\infl(ya))=e_1 \cdots e_k\pt(\vend(y),\vst(a))$.

Suppose further that $ya \in \Nf$; then $\infl(ya) \in \widetilde \Nf$.
If $\epair(y,a)=0$, then $\suf_{\vst(a)}(y)=y \in \Nf_0$, $a \in
A_{v_0}$, and $ya \in \Nf_0$, and so $\lbl(\Phi(e_{g,a}))=\phi(y,a)=
\phi_0(y,a)=\lbl(\Phi_0(e_{y,a}))=a$ since the flow function
$\Phi_0$ fixes edges in the associated spanning tree for $G_{v_0}$.
If instead $\epair(y,a) \neq 0$, then $\epair(y,a) \in \velam$.  In
the case that $\vend(y)=\vst(a)$ (i.e., $a\in A_{\vend(y)}$), we have
$\epair(y,a)=e_k$ and $\suf_{\vst(a)}(y)=t_k$. Applying the
properties of elements of $\widetilde \Nf$ to $\infl(ya)$ shows that
$t_ka \in \Nf_{\Tr,e_k}$; that is,
$\suf_{\vst(a)}(y)a \in \Nf_{\Tr,\epair(y,a)}$. If
$\suf_{\vst(a)}(y)=t_k=\emptyword$, then (recalling that
$\Nf_{\Tr,e_k} = \{\emptyword\} \cup
    (\Nf_{G_{\vend(e_k)},e_k} \cap
    (A_{\vend(e_k)} \setminus B_{e_k}) \cdot A_{\vend(e_k)}^*)$)
we have $a \notin B_{\epair(y,a)}$. Hence
$\lbl(\Phi(e_{g,a}))=\phi_{\epair(y,a)}(\suf_{\vst(a)}(y),a)=a$,
using Property (F2) for the flow function $\Phi_{\epair(y,a)}$. In
the case that $\vend(y) \neq \vst(a)$, the edge $\epair(y,a)$ is the
last edge of $\pt(\vend(y),\vst(a))$ and
$\suf_{\vst(a)}=\emptyword$. Again using the fact that $\infl(ya)
\in \widetilde \Nf$, the edge $\epair(y,a)$ must be followed in
$\infl(ya)$ by a suffix in the normal form set
$\Nf_{\Tr,\epair(y,a)}$; this subword is the single letter $a$, and
so $a \in \Nf_{\Tr,\epair(y,a)}$ and again $a \notin
B_{\epair(y,a)}$. Therefore
$\lbl(\Phi(e_{g,a}))=\phi_{\epair(y,a)}(\suf_{\vst(a)}(y),a)=a$, as
needed.

Suppose instead that $y$ ends with the letter $a^{-1} \in A_{\vst(a)}$.
Then $\vend(y)=\vst(a)$ and
$\suf_{\vst(a)}(y)$ is a nonempty word in $A_{\vst(a)}^*$
ending with the letter $a^{-1}$.  Therefore
$\lbl(\Phi(e_{g,a}))=\phi_{\epair(y,a)}(\suf_{\vst(a)}(y),a)=a$.
This concludes the last case, showing
that $\Phi$ fixes every directed edge lying
in the spanning tree $T$ in $\Gamma$, and so (F2) is satisfied.

\noindent{\bf Property (F3):} Finally we consider sequences
$e_1,e_2,...$ of directed edges $e_i \in \vec E$ in $\Gamma$ that do
not lie in $T$, satisfying the property that $e_{i+1}$ lies in the
path $\Phi(e_i)$ for each $i$.  To show that no infinite sequence of
this form can exist, we define a function $\counter:\vec E \ra \N^2$
and show that $\counter(e)>\counter(e')$ (using the lexicographic
order on $\N^2$) whenever $e,e' \in \vec E$ do not lie in the tree
$\tree$, and $e'$ is on the path $\Phi(e)$, as follows.

For each edge
$e_{g,a}^0$ with $g \in G_{v_0}$ and $a \in A_{v_0}$
(in the Cayley graph $\Gamma_{v_0}$
of $G_{v_0}$ over $A_{v_0}$),
the \emph{descending chain length} of $e_{g,a}^0$,
denoted $\dcl_0(e_{g,a}^0)$, is defined to be the
maximum possible number of edges of $\Gamma_{v_0}$
in a sequence
$e_{g,a}^0=e_1^0,e_2^0,...$ such that
each $e_i^0$ is an edge of $\Gamma_0$ lying
outside of the spanning tree associated to
the flow function $\Phi_0$ and $e_{i+1}^0$ lies
on $\Phi_0(e_i^0)$ for all indices $i$.
Because $\Phi$ is a flow function, this maximum is finite.

For each letter $f \in \velam$,
let $\vec E_{\vend(f)}$ be the set of directed edges
$e_{g,a}^f$ (with $g \in G_{\vend(f)}$ and $a \in A_{\vend(f)}$)
in the Cayley graph $\Gamma_{\vend(f)}$
of $G_{\vend(f)}$ over $A_{\vend(f)}$.
Also let $\vec E_{sub,f}$ be the subset of $\vec E_{\vend(f)}$ of edges
$e_{h,b}^f$ satisfying that $h \in \homom_f(G_f)$
and $b \in B_f$; that is, $e_{h,b}^f$ is in the
subgraph of $\Gamma_{\vend(f)}$ that is the Cayley
graph for $\homom_f(G_f)$ over its generating set $B_f$.
For each edge $e_{g,a}^f$ in $\vec E_{\vend(f)}$, its
\emph{$f$--descending chain length},
denoted $\dcl_f(e_{g,a}^f)$, is defined to be the
maximum possible number of edges
in a sequence $e_{g,a}^f=e_1^f,e_2^f,...$ such that each $e_i^f$ is
an edge in $\vec E_{\vend(f)} \setminus \vec E_{sub,f}$ lying
outside of the spanning tree associated to the flow function
$\Phi_f$ and $e_{i+1}^f$ lies on $\Phi_f(e_i^f)$ for all indices
$i$. Note that for all $h \in \homom_f(G_f)$ the edge $e_{hg,a}^f$
also lies in $\vec E_{\vend(f)}$; the \emph{invariant
$f$--descending chain length} of $e_{g,a}^f$ is
$$
\idcl_f(e_{g,a}^f) := \min\{\dcl_f(e_{hg,a}^f) \mid h \in
\homom_f(G_f)\}.
$$
Note that for all edges $e_{h,b}^f \in \vec E_{sub,f}$,
we have $\dcl_f(e_{h,b}^f)=\idcl_f(e_{h,b}^f) = 0$, and
for all edges $e_{g,a}^f \in \vec E_{\vend(f)} \setminus \vec E_{sub,f}$,
we have $\dcl_f(e_{g,a}^f) \ge \idcl_f(e_{g,a}^f)  \ge 1$.


Recall that for any word $w$ over $A^*$,
the symbol $\ell(w)$ denotes the length of the
word $w$.  We now define
$\counter:\vec E \ra \N^2$ by
\begin{eqnarray*}
\counter(e_{g,a})
&:=&
(\counter_1(e_{g,a}),\counter_2(e_{g,a})) \text{ where}\\
\counter_1(e_{g,a}) &:=& \ell(\edgeonly(\infl(\nf{g}a))), \text{ and}\\
\counter_2(e_{g,a}) &:=& \begin{cases}
\dcl_{0}(e_{\nf{g},a}^{0}) &
    \text{if } \epair(\nf{g},a)=0 \\
\idcl_{\epair(\nf{g},a)}(e_{\suf_{\vst(a)}(\nf{g}),a}^{\epair(\nf{g},a)}) &
    \text{if } \epair(\nf{g},a) \in \velam. \end{cases}
\end{eqnarray*}

Let $e=e_{g,a}$ be any element of $\vec E$ that
does not lie in the spanning tree $\tree$, and let
$e'=e_{g',a'}$ be any edge on $\Phi(e)$ that also
is not in $\tree$.
Let $y:=\nf{g}$ and $y':=\nf{g'}$, and let $f:=\epair(y,a)$.

We consider the three cases
of the definition of $\phi(y,a)$ in turn.

\noindent{\em Case 1. Suppose that $a \in \velnot$.}
Since the word $y$ is in $\Nf$, then the inflation
$\infl(y)$ is an element of $\widetilde\Nf$.  Thus either the word
$\infl(y)\pt(\vend(y),\vst(a))a$
is also in $\widetilde\Nf$, in which case $ya \in \Nf$,
or else $\infl(y)$, and hence also the word $y$,
ends with $\overline{a}=a^{-1}$.
Then the edge $e=e_{g,a}$ lies in the spanning tree $T$ in $\Gamma$,
giving a contradiction.  So this case cannot occur.

\noindent{\em Case 2. Suppose that
$f:=\epair(y,a) \in \velam$,
$a \in B_{f}$, and $\suf_{\vst(a)}(y)=\emptyword$.}
Then $\phi(y,a)=\defl(f^{-1}~\widehat{a}_{f}~f)$
with $\widehat{a}_{f} \in A_{\vst(f)}^*$.
If $f \in \velnot$, then
the edges $e_{g,f^{-1}}$ and
$e_{gf^{-1}\widehat a_f,f}$ lying in the path $\Phi(e)$
are labeled by elements of $\velnot$, and hence (by Case~1)
lie in the spanning tree $\tree$ of the Cayley graph.
Consequently the edge $e'$ cannot be one of those two edges.
Thus there is a factorization
$\phi(y,a)=\defl(f^{-1})wa'w'\defl(f)$ for some words $w,w' \in A_{\vst(f)}^*$
such that $e'=e_{g',a'}$ and $g'=_G gf^{-1}w$.

Write  (as usual) $\infl(y)=w_0e_1t_1e_2t_2\cdots e_kt_k \in
\widetilde\Nf$ such that $e_1 \cdots e_k$ is an edge path in
$\Lambda$ from $v_0$ to $\vend(y)$, $w_0 \in \Nf_0$, each $t_i \in
\Nf_{\Tr,e_i}$, no inverse pair $e\overline{e}$ appears as a
subword, and the last letter of $\infl(y)$ (which is also the last
letter of $y$) lies in $A$. Write $\pt(\vend(y),\vst(a))=f_1 \cdots
f_m$ with $m \ge 0$ and each $f_j \in \velam \setminus \velnot$.
Then $\infl(ya)=\infl(y)\pt(\vend(y),\vst(a))a$ and so
$\counter_1(e)=\ell(\edgeonly(\infl(ya)))=k+m$.

Note that either $m>0$, in which case $f=f_m \in \velam \setminus \velnot$
and $\defl(f)=\emptyword$,
or else $m=0$, in which case $\vend(y)=\vst(a)$ and
$f=e_k \in \velnot$ is the last letter of $y$.
For the rest of Case~2 we consider the situation that
$m>0$; the proof for $m=0$ is similar.

Note that $y'=\nf{g'}=\nf{yw}=\nf{\infl(y)w}$.
Applying Lemma~\ref{lem:computenf} to the words
$\tilde{y}:=\infl(y)$ and $w \in  A_{\vst(f)}^*$
shows that
$$
\infl(y')=\mathsf{nf}_{\widetilde\Nf}(y')=
\trim(w_0'
    e_1x_{1}e_2 \cdots e_kx_{k}
    f_1 w_{1} f_2 \cdots f_{m-1} w_{m-1})
$$
for some words $w_0' \in \Nf_0$, $x_i \in \Nf_{\Tr,e_i}$
and $w_j \in \Nf_{\Tr,f_j}$.
Since the word being pruned away is freely reduced and
represents a path in $\tlam$, the
suffix that is removed is
$\pt(\vend(y'),\vend(f_{m-1}))$.
For the letter $a' \in A_{\vst(f)}=A_{\vend(f_{m-1})}$ we
have $\infl(y'a')=\infl(y')\pt(\vend(y'),\vst(f_{m-1}))$,
and so $\edgeonly(\infl(y'a'))=e_1 \cdots e_kf_1 \cdots f_{m-1}$;
that is, $\counter_1(e')=k+m-1$.  Therefore $\counter(e)>\counter(e')$
in this case.


\noindent{\em Case 3. Suppose that $a \in A_f$ and either $f=0$,
or else ($f \in \velam$ and either $\suf_{\vst(a)}(y) \neq
\emptyword$ or $a \notin B_f$).} This case is split into two
subcases.

\noindent{\em Case 3.1.  Suppose that $f=0$.} In this case the word
$y\in\Nf_0$ for the flow function $\Phi_0$ on the vertex group
$G_{v_0}$, and the letter $a\in A_{v_0}$, the generating set for
this vertex group. In this case the edge $e$ lies in the Cayley
graph $\Gamma_{A_{v_0}}(G_{v_0})$ of the vertex group (considered as
a subgraph of $\Gamma$), and
$\phi(y,a)=\phi_{0}(\suf_{\vst(a)}(y),a)=\phi_0(y,a)$. Then the edge
$e'$ on $\Phi(e)$ is also an edge on $\Phi_0(e)$, and so
$\alpha_1(e)=0=\alpha_1(e')$. Moreover, the descending chain lengths
of these edges satisfy $\dcl_0(e)>\dcl_0(e')$, and so
$\alpha_2(e)>\alpha_2(e')$.  Therefore $\alpha(e)>\alpha(e')$.

\noindent{\em Case 3.2.  Suppose that $f \in \velam$.} In this case
we again factor $\phi(y,a)=\phi_{f}(\suf_{\vst(a)}(y),a)=wa'w' \in
A_{\vend(f)}^*$ such that $e'=e_{g',a'}$ for the element $g'=_G yw$
and letter $a' \in \vend(f)$. We write
$\infl(y)=w_0e_1t_1e_2t_2\cdots e_kt_k$ and
$\pt(\vend(y),\vst(a))=f_1 \cdots f_m$ with $m \ge 0$,
and $\counter_1(e)=k+m$.  Applying Lemma~\ref{lem:computenf}
to the words
$\tilde{y}:=\infl(y)$ and $w,a' \in  A_{\vend(f)}^*=A_{\vend(f_{m})}^*$
shows that $\counter_1(e')=k+m=\counter_1(e)$ in this case.

Next consider $\counter_2(e)=\idcl_f(e_{\suf_{\vst(a)}(y),a}^f)$.
Since $f \neq 0$, then $\suf_{\vst(a)}(y)$ lies in $\Nf_{\Tr,f}$,
and so $\suf_{\vst(a)}(y)$ can only represent an element
of $\homom_{f}(G_{f})$ if $\suf_{\vst(a)}(y)=\emptyword$,
in which case $a \notin B_f$.  Hence in Case~3.2 we have
$e=e_{g,a}$ satisfies
$e_{\suf_{\vst(a)}(y),a}^f \in \vec E_{\vend(f)} \setminus \vec E_{sub,f}$,
and $\idcl(e_{\suf_{\vst(a)}(y),a}^f) \ge 1$.

Let $h$ be an element of $\homom_f(G_f)$ achieving the minimum for
this invariant descending chain length; that is,
$\counter_2(e)=\idcl(e_{\suf_{\vst(a)}(y),a}^f)=
\dcl(e_{h\suf_{\vst(a)}(y),a}^f)$. Since $\Phi_f$ is the bounded
flow function of an autostackable structure for $G_{\vend(f)}$ \rsp\
the subgroup $\homom_f(G_f)$, then $\Phi_f$ is
$\homom_f(G_f)$--translation invariant, and so for this element $h$,
we have $\lbl(\Phi_f(e_{h\suf_{\vst(a)}(y),a}^f))=wa'w'$ as well.
Then the edge $e_{h\suf_{\vst(a)}(y)w,a'}^f$ lies on the path
$\Phi_f(e_{h\suf_{\vst(a)}(y),a}^f)$. Now the descending chain
lengths satisfy
$\dcl(e_{h\suf_{\vst(a)}(y)w,a'})<\dcl(e_{h\suf_{\vst(a)}(y),a})=
\counter_2(e)$.

In order to compute $\counter_2(e')$, we first note that
Lemma~\ref{lem:computenf} also shows that $\epair(y',a')=f$ and so
$\counter_2(e')=\idcl_f(e_{\suf_{\vst(a)}(y'),a'}^f)$. However, by
definition,
$\idcl_f(e_{\suf_{\vst(a)}(y'),a'}^f)\le\dcl(e_{\suf_{\vst(a)}(y'),a'}^f)<\dcl(e_{h\suf_{\vst(a)}(y),a})=\counter_2(e)$,
as desired.
%
Hence, $\counter(e')<\counter(e)$ in this last case as well.

\medskip

\noindent\textbf{Autostackability:}

Next we show that the graph of the stacking function
$\phi$ associated to the flow function $\Phi$ is
a regular subset of $(A^*)^3$
in the case that the flow function $\Phi_0$
and each of the flow functions
$\Phi_f$ associated to the directed edges $f \in \velam$
gives an autostackable structure; that is,
the sets
\begin{eqnarray*}
\Graph{\Phi_0}&=&\{(y,a,\phi_0(y,a)) \mid y \in \Nf_0, a \in A_{0}\}
\text{ and} \\
\Graph{\Phi_f}&=&\{(y,a,\phi_f(y,a)) \mid y \in \Nf_f, a \in A_{\vend(f)}\}
\end{eqnarray*}
 are regular.

We begin by noting that Lemma~\ref{lemma:respnf} and
Proposition~\ref{prop:gognf} together show that the set
$\Nf$ of normal forms associated to the spanning
tree in $\Gamma$ for $\Phi$ is a regular language,
and moreover the set $\widetilde{\Nf}$ is also
a regular language of normal forms for $\pi_1(\gog)$,
with $\Nf=\defl(\widetilde{\Nf})$.

We proceed by breaking down the graph of $\Phi$ using the three
cases in the piecewise definition of its stacking function $\phi$:
\begin{eqnarray*}
\Graph{\Phi} &=&
\left( \cup_{a \in \velam \setminus \velnot}~
    \Nf \times \{a\} \times \{a\}\right) \\
&& \bigcup \left(\cup_{f \in \velam} \cup_{a \in B_f}~
   L_{f,a} \times \{a\} \times \{\defl(f^{-1}\widehat{a}_f f)\}\right) \\
&& \bigcup \left(\cup_{f \in \velam}\cup_{a \in A_{\vend(f)}}
   \cup_{w \in \mathsf{im}(\phi_f)}
   L_{f,a,w}' \times \{a\} \times \{ w \} \right)\\
&& \text{ } \hspace{.2in}  \bigcup \left(\cup_{a \in A_{v_0}}
   \cup_{w \in \mathsf{im}(\phi_0)}
   L_{0,a,w}' \times \{a\} \times \{ w \}\right)
\end{eqnarray*}
where
\begin{eqnarray*}
L_{f,a} &:=& \{ y \in \Nf \mid \epair(y,a)=f \text{ and }
  \suf_{\vst(a)}(y)=\emptyword\}, \\
L_{f,a,w}' &:=& \{ y \in \Nf \mid \epair(y,a)=f,~
   \suf_{\vst(a)}(y) \neq \emptyword \text{ and }
   \phi_f(\suf_{\vst(a)}(y),a)=w \}
  \text{ if } a \in B_f, \\
L_{f,a,w}' &:=& \{ y \in \Nf \mid \epair(y,a)=f \text{ and }
  \phi_f(\suf_{\vst(a)}(y),a)=w\}
  \text{ if } a \in A_{\vend(f)} \setminus B_f,  \text{ and }\\
L_{0,a,w}' &:=& \{ y \in \Nf \mid \epair(y,a)=0 \text{ and }
  \phi_0(\suf_{\vst(a)}(y),a)=w\}.
\end{eqnarray*}

Since the graph $\Lambda$ is finite
and each of the flow functions $\Phi_0$ and $\Phi_f$ are
bounded, this decomposition of $\Graph{\Phi}$ is a
finite union of subsets. The closure properties
of regular languages in Theorem~\ref{thm:regclosure}
imply that in order to show that $\Graph{\Phi}$ is regular,
it suffices to show that the languages
$L_{f,a}$ and $L_{f,a,w}'$ are regular.

Let $f \in \velam$ and $a \in A_{\vend(f)}$.
Since $\epair(y,a)$ is the last letter of the image
of $\infl(ya)$ under the monoid homomorphism
$\edgeonly$, any word $y$
in the normal form set $\Nf$ satisfying $\epair(y,a)=f$
either satisfies $\vend(y)=u$ for some vertex $u$ such that
$\pt(u,\vst(a))$ ends with $f$,
or else has inflation with a suffix in $f A_{\vend(f)}^*$.
Let $V_f$ be the set of vertices $u$ of $\Lambda$
such that the last edge of the path $\pt(u,\vst(a))$
is $f$.  For each vertex $v$ of $\Lambda$,
let $C_{v}$ be the (finite) set of all letters $c$ in $A$
with $\vend(c)=v$; that is, $C_v$ is the union
of $A_v$ with all of the edges in $\velam \setminus \velnot$
whose terminal vertex is $v$.
Then the set 
of normal form words $y$ with
$\epair(y,a)=f$ is
$$
R_{f,a}:=\{y \in \Nf \mid \epair(y,a)=f\}=
\left(\cup_{u \in V_f} \Nf \cap A^*C_u \right)
\bigcup R_{f,a}'
\bigcup \defl(\widetilde{\Nf} \cap A^*f A_{\vend(f)}^*)
$$
where $R_{f,a}'=\{\emptyword\}$ if $\pt(v_0,\vst(a))$ ends with the
letter $f$, and $R_{f,a}'=\emptyset$ otherwise. Since $\Nf$,
$\widetilde{\Nf}$, and $R_{f,a}'$ are regular, as are the
concatenations $A^*C_u$ and $A^* f A_{\vend(f)}^*$,
Theorem~\ref{thm:regclosure} implies that $R_{f,a}$ is also a
regular language.

Next suppose that $a \in A_{\vst(a)}$.  The set 
of normal form words $y$ whose maximal suffix
in $A_{\vst(a)}^*$ is empty is
$$
S_a:=\{y \in \Nf \mid \suf_{\vst(a)}(y)=\emptyword\} =
\{\emptyword\} \cup (\Nf \cap A^*(A \setminus A_{\vst(a)})).
$$
This union of a finite set with
an intersection  of regular languages is regular.
Now for all $f \in \velam$ and $a \in B_f$, the
language $L_{f,a}$ is the intersection
$
L_{f,a}=R_{f,a} \cap S_a,
$
and hence $L_{f,a}$ is regular.

Finally, suppose that
either $f \in \velam$ and $a \in A_{\vend(f)}$
or else $f=0$ and $a \in A_{v_0}$, and
let $w$ be a word in the image of $\phi_f$.
The set of normal forms $y$ for which the
image of the edge $e_{\suf_{\vst(a)}(y),a}$ under
$\Phi_f$ has label $w$ is
$$
Q_{f,a,w}:=\{y \in \Nf \mid \phi_f(\suf_{\vst(a)}(y),a)=w\}=
   \Nf \cap \left(S_a \cdot p_1(\Graph{\Phi_f} \cap
   (A_{\vst(a)}^* \times \{a\} \times \{w\}))\right)
$$
where $p_1$ denotes projection on the first coordinate.
Again applying Theorem~\ref{thm:regclosure} and
closure properties of regular languages,
the set $Q_{f,a,w}$ is regular.

Now if $a \in B_f$ then $L_{f,a,w}'=R_{f,a} \cap (\Nf\setminus S_a)
\cap Q_{f,a,w}$, and if $a \in A_{\vend(f)} \setminus B_f$ then
$L_{f,a,w}'=R_{f,a} \cap Q_{f,a,w}$. Finally if $f=0$ and $a \in
A_{v_0}$, then $\epair(y,a)=0$ for a normal form $y$ if and only if
$y$ lies in the regular language $\Nf_0$, and so $L_{0,a,w}'=\Nf_0
\cap Q_{0,a,w}$. Thus in all three cases $L_{f,a,w}'$ is regular.

Therefore $\Graph\Phi$ is
regular and $\pi_1(\gog)$ is autostackable.
\end{proof}



\section{Extensions and autostackability respecting subgroups}\label{sec:extn}


In this section
we record two results on autostackability {\rsp} subgroups,
for finite extensions and finite index supergroups,
which will be used in Section~\ref{sec:3mfld} in our
analysis of Seifert-fibered pieces of 3-manifolds.
The first uses the proof
of the closure of autostackability under group
extensions in~\cite[Theorem~3.3]{BHJ:closure}.

\begin{thm}\label{thm:closureextn}
Let
$
1   \rightarrow K \overset{i}{\rightarrow}G \overset{q}{\rightarrow} Q \rightarrow 1
$
be a short exact sequence of groups and group homomorphisms,
and let $H$ be a subgroup of $G$ containing $K$.
If $Q$ is autostackable [respectively, stackable] {\rsp} $q(H)$ and
$K$ is autostackable, then $G$ 
is also autostackable [respectively, stackable] {\rsp} $H$.
\end{thm}

\begin{proof}
Let $\Nf_K$, $\Phi_K$, and $\phi_K$ be the regular
prefix-closed normal form
set, bounded flow function, and associated stacking
map for $K$ over a (finite inverse-closed) generating
set $A_K$, and similarly let
$\Nf_Q=\Nf_{q(H)}\NTr$, $\Phi_Q$, and $\phi_Q$ be the
regular prefix-closed normal form
set, bounded flow function, and associated stacking
map for $Q$ over a generating set $C$ of $Q$, {\rsp}
the subgroup $q(H)$ with generating set $D \subset C$
and regular normal forms $\Nf_{q(H)}$.
By slight abuse of notation, we
will consider the homomorphism $i$ to be an
inclusion map, and $A,K,H \subseteq G$,
so that we may omit writing $i(\cdot)$.

For each $c \in C$, let $\hat c$ be an element of $G$ satisfying
$q(\hat c)=c$ (and choose these elements such that
$\hat{c^{-1}}=\hat{c}^{-1}$), and let $\hat C:=\{\hat c \mid c \in
C\}$. We note that $A_K \cap \hat C=\emptyset$, since an element of
the generating set $C$ for the autostackable structure on $Q$ does
not represent the identity in $Q$. Define the monoid homomorphism
$\hhh:C^* \ra \hat C^*$ by $\hhh(c):=\hat c$ for all $c \in C$. Let
$A := A_K \cup \hat C$ and $B:= A_K \cup \hat D$, and let
$$
\Nf_G:=\Nf_K \hhh(\Nf_Q)=\Nf_K \hhh(\Nf_{q(H)}) \hhh(\NTr).
$$
Then $\Nf_G$ is a prefix-closed regular language
of normal forms for the group $G$.

For any $g \in G$, write $\nf{g}=r_gs_gt_g$ with $r_g \in \Nf_K$,
$s_g \in \hhh(\Nf_{q(H)})$, and $t_g \in \hhh(\NTr)$;
similarly, for $y \in \Nf_G$, write
$y=r_ys_yt_y$ with $r_y \in \Nf_K$,
$s_y \in \hhh(\Nf_{q(H)})$, and $t_y \in \hhh(\NTr)$
For any nonempty word $w$, let $\last(w)$ denote
the last letter of $w$.

Define a function $\phi: \Nf_G \times A \rightarrow A^*$ by
$$
\phi(y, a) = \left\lbrace
\begin{array}{l l }
\phi_{K}(y,a)
  & \mbox{if } a \in A_K \mbox{ and }  s_yt_y=\emptyword
\\
\last(y)^{-1}r_{\last(y)a\last(y)^{-1}}\last(y)
  & \mbox{if }  a \in A_K \text{ and } s_yt_y \neq \emptyword
\\
r_{a (\hhh(\phi_Q(q(s_yt_y),q(a))))^{-1} }  {\mathsf{hat}}(\phi_Q(q(s_yt_y),q(a)))
  & \mbox{if } a \in \hat{C}.
\end{array} \right.
$$

Let $\Gamma:=\Gamma_A(G)$ be the Cayley graph of $G$ with respect to $A$,
and let $\vec E$ and $\vec P$ be the sets of directed edges
and directed paths in $\Gamma$.
Also define $\Phi:\vec E \ra \vec P$ by
$\Phi(e_{g,a}):=$ the path in $\Gamma$ starting
at $g$ and labeled by $\phi(\nf{g},a)$.

Now the proof of~\cite[Theorem~3.3]{BHJ:closure}
shows that $\Phi$ is a bounded flow function for
$G$ over $A$, and that the graph of $\Phi$ is regular;
that is, $\Phi$ gives an autostackable structure
for $G$ over $A$.

Let $T$ be the spanning tree of $\Gamma$ determined by the normal
form set $\Nf_G$. Since the normal form set $\Nf_G$ of this
structure is the concatenation of $\Nf_K \hhh(\Nf_{q(H)})$, which is
a set of normal forms for the subgroup $H$ over $B$, with the set
$\hhh(\NTr)$, which is a set of normal forms over $A$ for a right
transversal of $H$ in $G$, the tree $T$ of $\Gamma$ associated to
the set $\Nf_G$ is the union of a spanning tree in the Cayley
subgraph $\Gamma_{B}(H)$ with an $H$--orbit of a transversal tree
for $H$ in $G$, as required.

Suppose that $h \in H$ and $b \in B$,
and consider the edge $e:=e_{h,b}$ of $\Gamma_{B}(H)$
contained in $\Gamma$.  If $h \in K$ and $b \in A_K$,
then the first case of the definition of $\phi$
shows that $\lbl(\Phi(e)) \in A_K^* \subseteq B^*$.
If $h \in H \setminus K$ and $b \in A_K$,
then the second case of the definition of $\phi$
applies, and since $\last(\nf{h}) \in \hat D \subseteq B$
and $r_{\last(y)a\last(y)^{-1}} \in B^*$,
again the label of $\Phi(e)$ is in $B^*$.
Finally, if $b \in B \setminus A_K$, then the
third case applies.  In this case, since
$\Phi_Q$ arises from an autostackable structure
{\rsp} $q(H)$, and since $q(b) \in D$,
the word $\phi_Q(q(h),q(b))$ labeling
$\Phi_Q(e_{q(h),q(b)})$ in $\Gamma_{C}(Q)$
must be in $D^*$.  Then ${\mathsf{hat}}(\phi_Q(q(h),q(a))) \in \hat D^*$,
and since
$r_{a (\hhh(\phi_Q(q (y), q(a))))^{-1} }  \in A_K^*$, we have
$\lbl(\Phi(e)) \in B^*$.  Thus $\Phi$ satisfies the
subgroup closure property for $H$.

On the other hand, suppose that $e=e_{g,a}$
is an edge of $\Gamma$ that
does not lie in $\Gamma_B(H)$, and that $h \in H$.
If $a \in A_K$, then $g \notin H$, and so
$\nf{g} \notin \Nf_K \hhh(\Nf_{q(H)})$.
In this case, then $t_g \neq \emptyword$ and so the letter
$\ell:=\last(\nf{g})$ lies in $\hat C \setminus \hat D$.
Since $\nf{hg}=\nf{h}t_g$, then $\last(\nf{hg})=\ell$ as well,
and case~2 of $\phi$
shows that $\lbl(\Phi(e_{hg,a}))=
\ell^{-1}r_{\ell a\ell^{-1}}\ell=\lbl(\Phi(e))$.
Suppose instead that $a \notin A_K$.  Then either
$a \in \hat D$ and $g \notin H$, or else $a \in \hat C \setminus \hat D$.
Hence either $q(g) \notin q(H)$ or else $q(a) \notin D$,
and so the edge $e_{q(g),q(a)}$ is not in the
Cayley subgraph $\Gamma_D(q(H))$ of $\Gamma_C(Q)$.
For this case $\lbl(\Phi(e))=
r_{a (\hhh(\phi_Q(q(s_yt_y),q(a))))^{-1}}
    {\mathsf{hat}}(\phi_Q(q(s_yt_y),q(a)))$.
Note that $q(H)$--translation invariance of the autostackable
structure $\Phi_Q$ for $Q$ {\rsp} $q(H)$, together with the fact
that the $H$--coset representatives for $y$ and $hy$ satisfy
$t_y=t_{hy}$, implies that
$\phi_Q(q(s_yt_y),q(a))=\phi_Q(q(s_{hy}t_{hy}),q(a))$, and so
$\lbl(\Phi(e))=\lbl(\Phi(e_{hg,a}))$ in this case as well. Hence
$\Phi$ is also $H$--translation invariant.
\end{proof}

The next result is shown by the first two authors and
Johnson in~\cite{BHJ:closure}.  Although
the theorem as stated in that paper did not use the
phrase ``{\rsp} $H$'', the proof
does imply this extra property.

\begin{prop}\cite[Theorem~3.4]{BHJ:closure}\label{prop:closuresupergp}
Let $H$ be an autostackable [respectively, stackable] group,
and let $G$ be a group containing $H$ as a subgroup of finite index.
Then $G$ is autostackable [respectively, stackable] {\rsp} $H$.
\end{prop}


\section{Strongly coset automatic groups and relative hyperbolicity}\label{sec:cosetautgp}


In this section we show, in Theorem~\ref{thm:cosetautostack},
an extension to coset automaticity and autostackability
respecting autostackable subgroups,
of the result by the first two authors and
Holt~\cite[Theorem~4.1]{BHH:algorithms}
that every automatic group with respect to prefix-closed
normal forms is autostackable.
We then apply this result to relatively hyperbolic groups. See
Subsections~\ref{sub:cosetaut} and \ref{sub:relhyp} for notation and
definitions.


\begin{thm}\label{thm:cosetautostack}
Let $G$ be a finitely generated group
and $H$ a finitely generated autostackable subgroup of $G$.
If the pair $(G,H)$ is strongly prefix-closed coset automatic,
then $G$ is autostackable {\rsp} $H$.
\end{thm}

\begin{proof}
Since $(G,H)$ is strongly prefix-closed automatic,
Definition~\ref{def:cosetaut} says that there is an inverse-closed
generating set $C$ for $G$, a prefix-closed regular set $\NTr
\subseteq C^*$ of unique representatives of the right cosets $Hg$ of
$H$ in $G$, and a constant $K \ge 0$ such that the language $\NTr$
satisfies the $H$--coset $K$--fellow traveler property.
Autostackability of $H$ gives a bounded flow function $\Phi_H$ for
$H$ over a finite inverse-closed generating set $B$, with a bound
$K_H$, such that $\Graph{\Phi_H}$ is regular. Let $\Nf_H$ be the set
of normal forms for $H$ over $B$ that are the labels of the
non-backtracking paths in the spanning tree of the Cayley graph
$\Gamma_B(H)$ associated to $\Phi_H$, and let $\phi_H:\Nf_H \times B
\ra B^*$ be the stacking function obtained from $\Phi_H$. Then
$\Nf_H$ is prefix-closed, and since $\Nf_H$ is the projection on the
first coordinate of the regular language $\Graph{\Phi}$,
Theorem~\ref{thm:regclosure} shows that $\Nf_H$ is also regular.

By taking separate copies of letters representing
the same group element, if necessary, we may assume
that $B \cap C=\emptyset$.  Let $A:=B\sqcup C$
and let $\Gamma:=\Gamma_{A}(G)$ be the Cayley
graph of $G$ with respect to the generating set $A$.
Then the set
$$
\Nf_G:=\Nf_H\NTr
$$
is a prefix-closed regular language of normal forms for $G$ over
$A$.  Let $T$ be the spanning tree
in $\Gamma$ consisting of the edges that lie on
paths starting at $\groupid$ and labeled by
words in $\Nf_G$.

For any element $g \in G$, we can write its normal form in $\Nf_G$
uniquely as $\nf{g}=x_gz_g$ with $x_g\in\Nf_H$ and $z_g\in\NTr$.
Similarly, for each $y \in \Nf_G$, write $y=x_yz_y$ with
$x_y \in \Nf_H$ and $z_y \in \NTr$.

\noindent{\bf The flow function:}  Next we construct a function
$\phi:\Nf_G \times A \ra A^*$.
We begin with some more notation.

Fix a total ordering on $C$, and for any letter $b \in B$,
denote the shortlex least word over $C$ representing the same element
of $G$ as $b$ by $\mathsf{sl}_C(b)$.

Recall that for any word $w$ in $A^*$ and integer $i \ge 0$, the
symbol $w(i)$ denotes the prefix of $w$ of length $i$ if $\ell(w)\ge
i$ and $w(i)=w$ if $\ell(w) < i$. Let $w(i)'$ denote the
suffix of $w$ with the first $i$ letters removed;
that is, if $w=a_1 \cdots a_k$ with each $a_j \in A$, then
$w(i)=a_1 \cdots a_i$  and $w(i)':=a_{i+1} \cdots a_k$
if $\ell(w) > i$, and $w(i)':=\emptyword$ if $\ell(w) \le i$.
Note that $w=w(i)w(i)'$.  The symbol $\mathsf{red}(w)$
represents the resulting freely reduced word obtained from $w$ after
all subwords of the form $aa^{-1}$ are (iteratively) removed.

Also recall from Proposition~\ref{prop:strongimpliesweak} that for
each element $h \in H \cap B_{\Gamma_C(G)}(K)$ and $c \in C$, there
is a finite state automaton $M_{h,c}$ accepting the set of all pairs
$(z,z')$ with $z,z' \in \NTr$ and $zc=_G hz'$, with state set
$\widetilde Q$, initial state $(q_0,h)$, accept states $P \times
\{c\}$, and transition function $\widetilde \delta$.  Note that the
set of states $\widetilde Q$ and the transition function $\widetilde
\delta$ of $M_{h,c}$ in the proof of
Proposition~\ref{prop:strongimpliesweak} do not depend upon $h$ or
$c$.  Hence all of these automata have the same number of states; we
denote this number by $\mu$. On the other hand, the set $P \times
\{c\}$ of accept states does depend on $c$, although it is
independent of $h$. For each state $\widetilde q$ of $M_{h,c}$ for
which there is a path from $\widetilde q$ to an accept state $(p,c)$
of $M_{h,c}$ (viewing the finite state automaton as a graph with
labeled directed edges), there must also be a simple path of length
at most $\mu$ from $\widetilde q$ to an accept state. Such a state
is called a \emph{live} state of $M_{h,a}$.
Fix a choice of a pair of words
$(v_{c,\widetilde q},w_{c,\widetilde q}) \in C^* \times C^*$ such
that $\ell(v_{c,\widetilde q}),\ell(w_{c,\widetilde q}) \le \mu$ and
$\widetilde \delta(\widetilde q,(v_{c,\widetilde q},w_{c,\widetilde
q}))$ is an accept state of $M_{h,c}$.

Let $y \in \Nf_G$ and $a \in A$. The stacking function $\phi$ is
given by
$$
\phi(y,a):=
\begin{cases}
\phi_H(y,a)               & \text{ if } a \in B \text{ and }
                            z_y = \emptyword \\
\mathsf{sl}_C(a)          & \text{ if } a \in B \text{ and }
                      z_y \neq \emptyword \\
a          & \text{ if } a \in C \text{ and either }
                      \nf{ya}=ya \text{ or } y \in A^*a^{-1} \\

\mathsf{red}(z_y^{-1}x_{z_ya}z_{z_ya})
                          & \text{ if } a \in C,~\nf{ya} \neq ya,~
                            y \notin A^*a^{-1} \text{ and }
                            \ell(z_y) \le \mu \\
(z_y(j)')^{-1}
  v_{a,\widetilde q}a
  (w_{a,\widetilde q})^{-1}
  z_{z_ya}(j)'
                   & \text{ if } a \in C,~\nf{ya} \neq ya,~
                     y \notin A^*a^{-1}, \ell(z_y) > \mu,\\
                   & \text{ }~~j:=\ell(z_y) - \mu - 1 , \text{ and } 
       \widetilde q:=\widetilde \delta((q_0,x_{z_ya}),(z_y(j),z_{z_ya}(j))).
\end{cases}
$$
Let $\Phi:\vec E \ra \vec P$ be defined by $\Phi(e_{g,a}) :=$ the
path in $\Gamma_{A}(G)$ starting at $g$ labeled by $\phi(\nf{g},a)$,
for all $g \in G$ and $a \in A$.

\noindent{\bf Property (F1):} It follows immediately from this definition that
$\Phi(e_{g,a})$ has the same initial and terminal vertices as
$e_{g,a}$ for all $g \in G$ and $a \in B$.  Suppose instead that $a \in C$.
If $\ell(z_g) \le \mu$, then since $z_g a =_G x_{z_ga}z_{z_ga}$,
again $\Phi$ fixes the endpoints of $e_{g,a}$.

On the other hand, suppose that $\ell(z_g) \ge \mu+1$ and let
$y:=\nf{g}$ and $j:=\ell(z_y)-\mu-1$. In this case since $a$ is a
single letter in $C$ and $z_y,z_{z_ya} \in \NTr$ satisfy $z_y a =_G
x_{z_ya}z_{z_ya}$, the $H$--coset $K$--fellow traveler property
implies that the element $h$ of $H$ represented by $x_{z_ya}=x_h$
lies in $B_{\Gamma_C(G)}(K)$. Hence the (padded word corresponding
to the) pair $(z_y,z_{z_ya})$ is accepted by the automaton
$M_{h,a}$, and the state $q_{y,a}:=\widetilde
\delta((q_0,h),(z_y,z_{z_ya}))$ is an accept state of $M_{h,a}$.
Factor the word
$z_y$ as $z_y=z_y(j)z_y(j)'$, and note that $z_y(j)'$ is the suffix
of $z_y$ of length $\mu+1$. Similarly factor
$z_{z_ya}=z_{z_ya}(j)z_{z_ya}(j)'$.
Now the state
$\widetilde q:=\widetilde \delta((q_0,h),(z_{y}(j),z_{z_ya}(j)))$
of the automaton $M_{h,a}$ satisfies
$\widetilde \delta(\widetilde q,(z_y(j)',z_{z_ya}(j)'))=
q_{y,a}$, and so
 $\widetilde q$ is live in $M_{h,a}$.
Now the pair $(z_y(j)v_{a,\widetilde q},z_{z_ya}(j)w_{a,\widetilde q})$
is accepted by $M_{h,a}$, and so we have
$z_y(j)v_{a,\widetilde q}a =_G hz_{z_ya}(j)w_{a,\widetilde q}
=_G x_{z_ya}z_{z_ya}(j)w_{a,\widetilde q}$.
Hence $(z_y(j)')^{-1}
  v_{a,\widetilde q}a
  (w_{a,\widetilde q})^{-1}
  z_{z_ya}(j)'=_G a$,
and $\Phi$ fixes the endpoints of $e_{g,a}$ in this
last case as well.

To see that the function $\Phi$ is bounded, we inspect each of the
cases. In the first case of the piecewise definition of $\phi$
above, the length of the path $\Phi(e_{g,a})$ is at most the bound
$K_H$ of the flow function $\Phi_H$, is it at most
$\max\{\ell(\mathsf{sl}_A(b)) \mid b \in B\}$ in the second, 1 in
the third, and $\max\{\ell(z^{-1}x_{za}z_{za}) \mid z \in \NTr,
\ell(z) \le \mu, \text{ and }a \in C\}$ in the fourth case.  Since
the two maxima are over finite sets, these are finite numbers. Now
suppose that the fifth case holds. Since the set $\NTr$ contains
only one representative of each coset, there is a unique word $z'$
such that $(z_y,z')$ is accepted by $M_{h,a}$, namely $z'=z_{z_ya}$,
and so the path in $M_{h,a}$ from $\widetilde q$ to $q_{y,a}$
labeled by $(z_y(j)',z_{z_ya}(j)')$ cannot have length greater than
$\ell(z_y(j)')+\mu$, since it cannot repeat a state after the word
$z_y(j)'$ is completed. To see this,
if a state is repeated, then the definition of the transition
function implies that
$z_{ya}(i)^{-1}x_{z_ya}^{-1}z_y=z_{ya}(k)^{-1}x_{z_ya}^{-1}z_y$ for
$j\le i\le k$. But then $z_{ya}(k)$ and $z_{ya}(i)$ represent the same
$H$--coset; however, $\NTr$ is prefix-closed and has unique coset
representatives. Thus $i=k$. Therefore $\ell(z_{z_ya}(j)') \le
2\mu+1$. Hence in this fifth case the length of $\Phi(e_{g,a})$ is
at most $(\mu+1)+\mu+1+\mu+(2\mu+1)=5\mu+3$. Therefore (F1) holds
for $\Phi$.

\noindent{\bf Property (F2):} Since a directed edge $e_{g,a}$
from $g \in G$ labeled by $a \in A$ in $\Gamma$
lies in the spanning tree $T$ obtained from $\Nf_G$
if and only if either $\nf{ga}=\nf{g}a$ or $\nf{g} \in A^*a^{-1}$,
it is immediate from the definition of $\Phi$ that
any edge $e_{g,a}$ in $T$ with $a \in C$ satisfies
$\Phi(e_{g,a})=e_{g,a}$.  Suppose instead that
$e_{g,a}$ is in $T$ and
$a \in B$.  Since
$\Nf_G=\Nf_H\NTr$ with $\Nf_H \subset B^*$ and $\NTr \subset C^*$,
we must have $z_g=\emptyword$, and so
the fact that the flow function $\Phi_H$ satisfies property (F2)
implies that $\Phi(e_{g,a})=e_{g,a}$ in this case as well.

\noindent{\bf Property (F3):}
In order to show that there is no infinite
sequence $e_1,e_2,... \in \vec E$ of directed edges
of $\Gamma_A(G)$ lying outside of the spanning tree $T$
such that $e_{i+1}$ is in the path $\Phi(e_i)$ for each $i$, we use
the same technique as in the proof of
Theorem~\ref{thm:GoGautostack}.  That is, we define a function
$\alpha:\vec E \ra \N^3$, and show that whenever $e,e' \in \vec E$
are not in $T$ and $e'$ is on the path $\Phi(e)$, then
$\alpha(e)>\alpha(e')$ (using the lexicographic order on $\N^3$).

Since $B$ is a subset of the generating set $A$ of $G$,
we can consider the Cayley graph $\Gamma_B(H)$ to be a
subgraph of the graph $\Gamma_A(G)$;
for each $h\in H$ and $b\in B$ we consider the
edge $e_{h,b}$ to be an edge of both of these graphs.
Let $\dcl_H(e_{h,b})$ be
the descending chain length for that edge from the autostackable
structure on $H$ (that is, the
maximum possible number of edges of $\Gamma_B(H)$
in a sequence
$e_{h,b}=e_1,e_2,...$ such that
$e_i$ is not in $T$ and $e_{i+1}$ is
on $\Phi_H(e_i)$ for all $i$).

For $1 \le i \le 3$ we define
functions $\alpha_i\colon \vec E\to \N$ by
\begin{eqnarray*}
\alpha_1(e_{g,a}) &:=&
\begin{cases}
\ell(z_g), & \text{if } a\in C\\
\max\{\ell(z_{g\mathsf{sl}_C(a)(i)}) \mid i < \ell(\mathsf{sl}_C(a))\}+1
           & \text{if } g\not\in H \text{ and } a\in B\\
0,         & \text{if } g\in H~,
\end{cases} \\
\alpha_2(e_{g,a}) &:=&
\begin{cases}
1, & \text{if } g \in H \text{ and } a\in C \\  
0, & \text{otherwise},
\end{cases}\\
\alpha_3(e_{g,a}) &:=&
\begin{cases}
\dcl_H(e_{g,a}), & \text{if } g\in H \text{ and } a\in B\\
0, & \text{otherwise}.
\end{cases}
\end{eqnarray*}
Now, let $\alpha=(\alpha_1, \alpha_2, \alpha_3)$.

Let $e=e_{g,a} \in \vec E$ be an edge outside of the
tree $T$, and let $e'=e_{g',a'}$ be an edge on $\Phi(e')$
that also is not in $T$.  Let $y:=\nf{g}$ and $y':=\nf{g'}$.
Consider the five cases of the definition of $\phi(y,a)$ in turn.

\noindent{\em Case 1.  Suppose that $a \in B$ and $z_y = \emptyword$.}
In this case $g \in H$, $\alpha(e)=(0,0,\dcl_H(e))$, and
$\lbl(\Phi(e))=\phi_H(y,a) \in B^*$.
Then  $a' \in B$ and $y'=_Gyw$ for a prefix $w$ of $\phi_H(y,a)$,
so $g' \in H$ and $\alpha(e')=(0,0,\dcl_H(e'))$.
Since $e'$ lies on $\Phi_H(e)$, the descending chain lengths
satisfy $\dcl_H(e)>\dcl_H(e')$, and so $\alpha(e)>\alpha(e')$.

\noindent{\em Case 2.  Suppose that $a \in B$ and $z_y \neq \emptyword$.}
In this case $g \notin H$ and
$\alpha(e)=
(\max\{\ell(z_{g\mathsf{sl}_C(a)(i)}) \mid i < \ell(\mathsf{sl}_C(a))\}+1,0,0)$.
We can factor $\lbl(\Phi(e))=\mathsf{sl}_C(a)=wa'w'$ for
some $w,w' \in C^*$ such that $g'=_G gw$ and $w=\mathsf{sl}_C(a)(i)$
for some $i<\ell(\mathsf{sl}_C(a))$.
Note that $a' \in C$.  Now either $g' \in H$,
in which case $\alpha_1(e')=0<\alpha_1(e)$, or else
$g' \notin H$, in which case
$\alpha_1(e')=\ell(z_{g'})=\ell(z_{g\mathsf{sl}_C(a)(i)}) <\alpha_1(e)$.
Hence in both options $\alpha(e)>\alpha(e')$.

\noindent{\em Case 3.  Suppose that $a \in C$ and either
                      $\nf{ya}=ya$ or $y \in A^*a^{-1}$.}
In this case $e$ lies in the tree $T$, so this case can't occur.

\noindent{\em Case 4.  Suppose that $a \in C$,
                      $\nf{ya} \neq ya$, $y \notin A^*a^{-1}$, and
                       $\ell(z_y) \le \mu$.}
In this case $\alpha(e)=(\ell(z_g),\alpha_2(e),0)$; if $g \in H$,
then $\alpha(e)=(0,1,0)$, and if $g \notin H$ then $\alpha_1(e)>0$.
We also have $\lbl(\Phi(e))=\mathsf{red}(z_y^{-1}x_{z_ya}z_{z_ya})$.
We can again factor the word $z_y^{-1}x_{z_ya}z_{z_ya}=wa'w'$ with
$w,w' \in A^*$ and $g' =_G gw$. Note that since $x_{ya}z_{ya} =_G ya
=_G x_yz_ya =_G x_{y}x_{z_ya}z_{z_ya}$, we have $x_{ya} =_H
x_{y}x_{z_ya}$ and $z_{ya}=z_{z_ya}$ since each right coset of $H$
in $G$ has only one representative in $\NTr$. If $a' \in C$, then
the edge $e'$ is either on the subpath starting at $g$ labeled by
$z_y^{-1}$, or the subpath ending at $ga$ labeled by
$z_{z_ya}=z_{ya}$; but these two paths are in the normal form tree
$T$, giving a contradiction.
So, we must have $a' \in B$
and $e'$ is on the subpath of $\Phi(e)$ starting at $gz_y^{-1}=_G
x_y \in H$ labeled by $x_{z_ya} \in B^*$. Hence $g' \in H$, and so
$\alpha(e')=(0,0,\dcl(e'))$. Therefore $\alpha(e)>\alpha(e')$.

\noindent{\em Case 5.  Suppose that $a \in C$,
                      $\nf{ya} \neq ya$, $y \notin A^*a^{-1}$, and
                       $\ell(z_y) > \mu$.}
In this case $\alpha(e)=(\ell(z_g),0,0)$.
Let $j:=\ell(z_y) - \mu - 1$ and
$\widetilde q:=\widetilde \delta((q_0,x_{z_ya}),(z_y(j),z_{z_ya}(j)))$; then
$\lbl(\Phi(e)=(z_y(j)')^{-1}  v_{a,\widetilde q}a  (w_{a,\widetilde q})^{-1}
  z_{z_ya}(j)'$.
As in Case~4, the subpath of $\Phi(e)$ starting at
$g$ labeled $(z_y(j)')^{-1}$, and the subpath ending at $ga$ and labeled
$z_{z_ya}(j)'=z_{ya}(j)'$, both lie in the spanning tree $T$.
Moreover, since the pair
$(z_y(j)v_{a,\widetilde q},z_{ya}(j)w_{a,\widetilde q})$
is accepted by the automaton $M_{h,a}$ for
$h=_G x_{z_ya}$, 
the words $x_yz_y(j)v_{a,\widetilde q}$ and $x_{ya}z_{ya}(j)w_{a,\widetilde q}$
are also normal forms in $\Nf_G$, and so the edges in $\Phi(e)$
in the subpaths labeled $v_{a,\widetilde q}$ and $(w_{a,\widetilde q})^{-1}$
also lie in the tree.  So we must have $a'=a$ and
$g'=_G g(z_y(j)')^{-1}  v_{a,\widetilde q}$.
Then $a' \in C$, and $z_{g'}=z_y(j)v_{a,\widetilde q}$.
Now $\alpha_1(e')=\ell(z_{g'})=\ell(z_y(j)v_{a,\widetilde q})=
\ell(z_y) - \mu - 1 +\ell(v_{a,\widetilde q}) < \ell(z_y)$
since $\ell(v_{a,\widetilde q}) \le \mu$.  Hence
$\alpha(e)>\alpha(e')$.

Then all of the properties (F1)--(F3) hold, and $\Phi$ is a flow
function for $G$ over $A$.

\medskip

\noindent\textbf{Respecting the subgroup $H$:} The definition of the
normal form set $\Nf_G$ as the concatenation $\Nf_H\NTr$ of normal
forms for $H$ over $B$ and normal forms for $H \backslash G$ over
$C$ implies the required structure for the spanning tree built from
these normal forms. Further, subgroup closure of the flow function
$\Phi$ is immediate from the first case of the definition of $\phi$.
The $H$--translation invariance of $\Phi$ follows from the fact that
the word $\lbl(\Phi(e_{g,a}))=\phi(\nf{g},a)$ depends only upon the
coset representative $z_g$ from the transversal, and is independent
of $x_g$, in the other four cases of the definition of $\phi$.

\medskip

\noindent\textbf{Autostackability:}
Recall from the beginning of this proof that
the sets $\Nf_G$, $\Nf_H$, $\NTr$, and $\Graph{\Phi_H}$
are all regular.
Also recall that for any natural number $j$,
the symbol $C^{\le j}$ denotes the set of all words over $C$
of length at most $j$.

Using the state set and transition
function from the automata $M_{h,a}$
constructed in the proof of Proposition~\ref{prop:strongimpliesweak},
we build more finite state automata as follows.  For
every $\widetilde q \in \widetilde Q$ and $\widetilde P \subset \widetilde Q$,
let $M_{\widetilde q,\widetilde P}$ be the automaton
with state set $\widetilde Q$, start state $\widetilde q$,
accept state set $\widetilde P$, and transition function
 $\widetilde \delta$.  Let $L(M_{\widetilde q,\widetilde P})$
be the set of all words accepted by this finite state automaton.

In order to show that $\Graph{\Phi}$
is also regular,
we separate the graph of $\Phi$ into five pieces
using the five cases in the definition
of its stacking function $\phi$:
\begin{eqnarray*}
\Graph{\Phi} &=&
\Graph{\Phi_H} \\
&& \bigcup \left(\cup_{a \in B} (\Nf_G \setminus \Nf_H)
  \times \{a\} \times \{\mathsf{sl}_C(a)\}   \right) \\
&& \bigcup \left(\cup_{a \in C} L_a \times \{a\} \times \{a\}  \right) \\
&& \bigcup \left(\cup_{a \in C} \cup_{z \in \NTr\cap B^{ \le \mu}}
  L_{a,z}' \times \{a\} \times \{\mathsf{red}(z^{-1}x_{za}z_{za})\}
  \right) \\
&& \bigcup \left(\cup_{a \in C} \cup_{\widetilde q \in \widetilde Q}
         \cup_{(u,u') \in (C^{\mu+1} \times C^{\le 2\mu+1})
             \cap L(M_{\widetilde q,P \times \{a\}})}
 L_{a,\widetilde q,u,u'}'' \times \{a\} \times \{u^{-1}v_{a,\widetilde q}a
                                     (w_{a,\widetilde q})^{-1}u'\}
  \right)
\end{eqnarray*}
where
\begin{eqnarray*}
L_{a} &:=& \{ y \in \Nf_G \mid \nf{ya}=ya \text{ or } y \in A^*a^{-1}\}
  =(\Nf_G)_a \cup (\Nf_G \cap A^* a^{-1}), \\
L_{a,z}' &:=& \{ y \in \Nf_G \mid y \notin L_a \text{ and } y \in B^*z\}
  =(\Nf_G \setminus L_a) \cap B^*z,
     \text{ and }\\
L_{a,\widetilde q,u,u'}'' &:=& \{ y \in \Nf_G \mid y \notin L_a 
  \text{ and } \exists~z,z' \in C^*
  \text{ and } h \in H \cap B_{\Gamma_C(G)}(K)
  \text{ such that } \\
&& \text{ } \hspace{1in} y \in B^*zu \text{ and }
 \widetilde \delta((q_0,h),(z,z'))=\widetilde q \}.
\end{eqnarray*}
Note that although the word $u'$ does not appear in
the definition of the set $L_{a,\widetilde q,u,u'}''$,
it follows from the fact that $(u,u') \in L(M_{\widetilde q,P \times \{a\}})$
that if $\widetilde \delta((q_0,h),(z,z'))=\widetilde q$
for some $h \in H \cap B_{\Gamma_C(G)}(K)$ and
$z,z' \in C^*$, then $(zu,z'u') \in L(M_{(q_0,h),P \times \{a\}})$, and
$zua =_G hz'u'$ with $z'u' \in \NTr$, and
so uniqueness of coset representatives among the
words in $\NTr$ implies that $z$, $u$, $a$, and $z'$
uniquely determine $u'$.  Moreover, the equation $zua =_G hz'u'$
also shows that the element $h$ must satisfy $h=_G x_{zua}$; that is,
$h$ is also uniquely determined by $z$, $u$, and $a$.

The closure of regular languages under products and
unions implies that it suffices to show that the
languages $L_a$, $L_{a,z}$ and $L_{a,\widetilde q,u,u'}''$ are regular.
Closure of regular languages under quotients (Theorem~\ref{thm:regclosure}),
unions, and intersections shows that the language $L_a$ is regular, and
closure under complementation and concatenation shows that
$L_{a,z}'$ is also regular.

Analyzing the language $L_{a,\widetilde q,u,u'}''$ further,
we have
\begin{eqnarray*}
L_{a,\widetilde q,u,u'}'' &=& (\Nf_G \setminus L_a) \cap
(\cup_{h \in H \cap B_{\Gamma_C(G)}(K)}~
B^*L_{h,\widetilde q}'''~u) \hspace{.2in} \text{ where} \\
L_{h,\widetilde q}''' &:=& \{z \in C^* \mid \exists~z' \in C^* \text{ such that }
  \widetilde \delta((q_0,h),(z,z'))=\widetilde q\}.
\end{eqnarray*}
Now $L_{h,\widetilde q}''' =
p_1(L(M_{(q_0,h),\{\widetilde q\}}))$, where
$p_1$ denotes projection on the first coordinate.
Since $L(M_{(q_0,h),\{\widetilde q\}})$ is the language
of a finite state automaton, it is regular, and so
 closure under projection  
shows that $L_{h,\widetilde q}'''$ is regular.
Hence each set $L_{a,\widetilde q,u,u'}''$ is regular.

Therefore $\Graph{\Phi}$ is a regular language.
Thus, $G$ is autostackable {\rsp} $H$.
\end{proof}

\begin{rmk}\label{rmk:algorithm}
Suppose that $(G,H)$ is a strongly shortlex coset automatic pair
such that both of the groups $G$ and $H$ are also
shortlex automatic.
Holt and Hurt~\cite{HoltHurt:Coset} have shown that
there is an algorithm which, upon input of a finite
presentation of $G$ (over the relevant generating set)
and the finite generating set of $H$,
can compute the finite state automata
$M_{h,a}$ of Proposition~\ref{prop:strongimpliesweak},
together with a finite state automaton accepting
the shortlex transversal,
for the strongly shortlex coset automatic structure for $(G,H)$.
Since $H$ is shortlex automatic,
there is an algorithm to compute the shortlex automatic structure
for $H$ from a finite presentation with the associated generators as
well (see~\cite[Chapters~5--6]{wordprocessing} for more details).
Hence there also is an algorithm which can produce the automaton
accepting the regular language $\Graph{\Phi}$; that is, it is
possible to algorithmically compute the autostackable structure on
$G$ {\rsp} $H$ in this case.
\end{rmk}

For hyperbolic groups, we obtain the following corollary which will
be used in the following section.

\begin{cor}\label{cor:hypcyc}
Hyperbolic groups are autostackable
{\rsp} quasiconvex subgroups. In particular, a hyperbolic group
is autostackable {\rsp} any virtually cyclic subgroup.
\end{cor}

\begin{proof}
Let $G$ be a hyperbolic group and
let $H\le G$ be a quasiconvex subgroup. In \cite[Chapter~10]{Redfern:thesis},
Redfern proves that any hyperbolic group
has the coset fellow traveler property with respect to any
quasiconvex subgroup using the shortlex transversal for the
right cosets (over any finite
generating set, and with respect to any ordering
on that finite set).  Theorem~\ref{thm:shortlexcosetaut}
then shows that the pair $(G,H)$ is strongly shortlex coset
automatic.  Since $H$ is quasiconvex in $G$, then $H$ is hyperbolic
(see, for example,~\cite[Proposition~III.$\Gamma$.3.7]{Bridson:NPC}),
and so $H$ is autostackable by~\cite[Theorem~4.1]{BHH:algorithms}.
Now apply Theorem~\ref{thm:cosetautostack} to see that $G$ is
autostackable {\rsp} $H$.

The last claim follows since virtually
cyclic subgroups of a hyperbolic group are quasiconvex
(\cite[Corollaries~III.$\Gamma$.3.6,III.$\Gamma$.3.10]{Bridson:NPC}).
\end{proof}

As discussed in Section~\ref{sub:relhyp} above,
Antolin and Ciobanu~\cite[Corollary~1.8]{AntolinCiobanu:relhyp}
showed that groups that are hyperbolic relative to a collection of abelian
subgroups are shortlex biautomatic
using a ``nice'' generating set.
In the remainder of this section we extend their
argument to obtain strong shortlex coset automaticity
and autostackability of the group respecting any of its
peripheral subgroups.  This is
critical in our analysis of fundamental groups
of hyperbolic pieces for Section~\ref{sec:3mfld}.

\begin{thm}\label{thm:relhypauto}
Let $G$ be a group 
that is hyperbolic relative to a collection of subgroups $\{H_1,...,H_n\}$
and is generated by a finite set $A'$.
Suppose that for every index $j$, the group
$H_j$ is shortlex biautomatic on every finite ordered generating set.
Then there is a finite subset
$\HH' \subseteq \HH:=\cup_{j=1}^n (H_j \setminus \groupid)$ such that
for every finite generating set $A$ of $G$ with
$A' \cup \HH' \subseteq A \subseteq A' \cup \HH$ and any ordering on $A$,
and for any $1 \le j \le n$, the pair $(G,H_j)$ is strongly shortlex
coset automatic, and $G$ is autostackable {\rsp} $H_j$, over $A$.
\end{thm}

\begin{proof}
Theorem~\ref{thm:quasigeodesic} says that there are constants
$\lambda \ge 1$ and $\epsilon \ge 0$ and a finite subset $\HH'
\subseteq \HH$ such that any finite $A$ satisfying $A' \cup \HH'
\subseteq A \subseteq A' \cup \HH$ is a
$(\lambda,\epsilon)$--{\nice} generating set of $G$ with respect to
$\{H_1,...,H_n\}$. Let $B:=B(\lambda,\epsilon+\lambda+1)$ be the
bounded coset penetration constant from
Proposition~\ref{prop:bcpwithepsilon}.

Fix a total ordering on $A$.
Since for each index $j$ the set $A \cap H_j$ generates $H_j$
(from Definition~\ref{def:nice}(2)),
by hypothesis the group $H_j$ is shortlex biautomatic
on this ordered set.  Then {\nice}ness of $A$ implies
that the group $G$ is shortlex biautomatic on the
generating set $A$.  Let $K$ be the fellow traveler constant
associated to this biautomatic structure.

Now fix an index $j \in \{1,...,n\}$, and let $L_\Tr \subset A^*$ be
the set of shortlex least representatives of the right cosets
of $H_j$ in $G$.  Suppose that $v,w \in L_\Tr$,
$a \in A$, and $h \in H_j$ satisfy $va=_G hw$.

Let $p$ be the path in $\Gamma_A(G)$ starting at $\groupid$ labeled
by $v$, and let $q$ be the path in $\Gamma_A(G)$ starting at $h$
labeled by $w$.  Then $p$ and $q$ are geodesics, and so have no
parabolic shortenings. Consider the paths $\hat{p}$ and $\hat{q}$ in
$\Gamma_{A\cup\HH}$ derived from the paths $p$ and $q$.  By
Definition~\ref{def:nice}(1), both $\hat p$ and $\hat q$ are
$(\lambda,\epsilon)$--quasigeodesics without backtracking.  Since
the words $v,w$ are shortlex minimal in their right $H_j$ cosets, no
nonempty prefix of $v$ or $w$ can represent an element of $H_j$. As
a consequence, the paths $\hat p$ and $\hat q$ cannot penetrate the
coset $\groupid H_j$.

Let $e$ be the edge in $\Gamma_{A\cup\HH}$ from $\groupid$ to $h$
labeled by $h$; then the concatenation $e\hat q$ is a
$(\lambda,\epsilon+\lambda+1)$--quasigeodesic (see, for
example,~\cite[Lemma~3.5]{osin}) in $\Gamma_{A\cup\HH}$ from
$\groupid$ to $\vend(q)$. The edge $e$ is an $H_j$--component of the
path $e\hat q$ lying in the coset $\groupid H_j$. Since $\vst(\hat
p)=\vst(e\hat q)$ and $d_{\Gamma_A(G)}(\vend(\hat p),\vend(e\hat q))
\le 1$ all of the hypotheses of the bounded coset penetration
property are satisfied for the pair of paths $\hat p, e\hat q$, and
so by Proposition~\ref{prop:bcpwithepsilon} applied to
Definition~\ref{def:relhyp}(2)(a), the component $e$ satisfies
$d_{\Gamma_A(G)}(\vst(e),\vend(e)) \le B$. That is,
$d_{\Gamma_A(G)}(\groupid,h) \le B$.

Write $h=a_1 \cdots a_m$ with each $a_i \in A$ and $m \le B$.
For each $0 \le i \le m$, let $w_i \in A^*$ be the shortlex least
representative of $(a_1 \cdots a_i)^{-1}va$ in $G$, and let
$r_i$ be the path in $\Gamma_A(G)$
from the vertex $\vst(r_i)=_G a_1 \cdots a_i$
to the vertex $\vend(r_i)=_G ua$ labeled by $w_i$.
Now the paths $p,r_0$ have the same initial point
and terminal vertices that are a distance 1 apart
in $\Gamma_A(G)$, and so shortlex biautomaticity implies that these paths
$K$-fellow travel.  Similarly each pair of
paths $r_{i-1},r_i$ (with $1 \le i \le m$) start a distance 1 apart
and terminate at the same vertex, and so they
$K$-fellow travel.  Now $r_m$ is the original path $q$,
and so
the paths $p$ and $q$ must $\widetilde K$--fellow
travel, for the constant $\widetilde K=(B+1)K$. That is, for all $i
\ge 0$ we have $d_{\Gamma_A(G)}(v(i),hw(i)) \le \widetilde K$.

Hence the pair $(G,H_j)$ satisfies the $H_j$--coset $\widetilde
K$--fellow traveler property using the shortlex normal forms of the
cosets, and Theorem~\ref{thm:shortlexcosetaut} shows that $(G,H_j)$
is strongly shortlex coset automatic. Now
Theorem~\ref{thm:cosetautostack} shows that $G$ is also
autostackable {\rsp} $H_j$ on the same generating set $A$.
\end{proof}



Finally we consider the special case that $G$ is hyperbolic relative
to abelian subgroups.
As noted in Remark~\ref{rmk:toralhypisnice}, Holt has
shown that finitely generated abelian groups satisfy
the property that they are shortlex biautomatic
on every ordered generating set.
Then Theorem~\ref{thm:relhypauto} and
Remark~\ref{rmk:algorithm} give the following.

\begin{cor}\label{prop:relhypautostack}
Let $G$ be a finitely generated group
that is hyperbolic relative to a collection $\{H_1,...,H_n\}$
of abelian subgroups.
Then there is a finite inverse-closed generating set $A$ of
$G$ such that for any $1 \le j \le n$, the group $G$ is
autostackable {\rsp} $H_j$ over $A$.
Moreover, there is an algorithm which, upon input of
a finite presentation for $G$ with generators $A$
and a finite presentation for $H_j$ with generators $A \cap H_j$,
produces the  autostackable structure.
\end{cor}


\section{3-Manifolds}\label{sec:3mfld}


In this section we prove that the fundamental group of any 
connected, compact
3-manifold with incompressible toral boundary is autostackable. Our
proof follows the procedure from Thurston's Geometrization for
decomposing a 3-manifold, discussed in Section~\ref{sub:3mfld}.

We begin with an analysis of the autostackability
of Seifert fibered 3-manifolds with incompressible toral
boundary, that arise in the JSJ decomposition of
prime, compact, nongeometric 3-manifolds 
with incompressible toral boundary.

\begin{prop}\label{prop:seifertautostack}
Let $M$ be a compact Seifert fibered 3-manifold with
incompressible toral boundary.
Let $T$ be any component of $\partial M$,
and let $H$ be any conjugate of $\pi_1(T)$ in $\pi_1(M)$.
Then $\pi_1(M)$ is autostackable {\rsp} $H$.
\end{prop}

\begin{proof}
Note that if $\partial M=\emptyset$, then $M$ is closed
and geometric, and so $\pi_1(M)$
is autostackable by~\cite[Corollary~1.5]{BHH:algorithms}.
For the remainder of this proof, we assume that $\partial M \neq \emptyset$.

Let $X$ be the base orbifold of the Seifert fibered space $M$,
and let $\pi_1^o(X)$ be the \emph{orbifold} fundamental group of $X$.
There exists a short exact sequence
$$
1\longrightarrow K \overset{i}\longrightarrow \pi_1(M)
\overset{q}\longrightarrow \pi_1^{o}(X)\longrightarrow 1,
$$
where $K \cong \Z$ is generated by a
regular fiber~\cite[Lemma~3.2]{Scott:3mflds}
and $K < \pi_1(T)$.
Since $K$ is normal in $\pi_1(M)$,
the conjugate $H$ of $\pi_1(T)$ also
satisfies $K < H$, and $H \cong \Z^2$.

Since the infinite cyclic group $K$ is autostackable,
Theorem~\ref{thm:closureextn} implies that in order to
prove that $\pi_1(M)$ is autostackable {\rsp} $H$,
it suffices to show that $\pi^o(X)$ is autostackable {\rsp} the image $q(H)$
of $H$ in this orbifold fundamental group.

Since M has nonempty boundary, the base orbifold X has nonempty
boundary, as
well. Comparing to the list of compact orbifolds in the
classification given in~\cite[Theorem~13.3.6]{thurston}
(and noting that there are no singular fibers over points in
the boundary of $X$), we find that no elliptic
or bad orbifold occurs as the base of a Seifert fibered space with
incompressible boundary, since all of these give a solid torus for $M$.
The only Euclidean (called parabolic in~\cite{thurston})
compact orbifolds which occur are the annulus,
the M\"obius band, and the 2-disk with two fibers
of multiplicity 2;
all other base orbifolds are hyperbolic.
We consider the Euclidean and hyperbolic cases separately.

Suppose first
that $X$ is a hyperbolic orbifold.
In this case the group $\pi_1^o(X)$ is hyperbolic. 
Since $H$ is abelian, the image $q(H)$ of $H$ in $\pi_1^o(X)$ is an
abelian subgroup of this hyperbolic group, and so $q(H)$ must be
virtually cyclic. Thus, by  Corollary~\ref{cor:hypcyc}, $\pi_1^o(X)$
is autostackable {\rsp}  $q(H)$, in any generating set for
$\pi_1^o(X)$ containing generators for $q(H)$.

Suppose instead that $X$ is a Euclidean orbifold. 
Then $X$ is one of the three possible orbifolds
listed above, all of which have orbifold fundamental group
$\pi_1^o(X)$ that is virtually $\Z$.
Since the kernel of the restriction of the map
$q:\pi_1(M) \ra \pi_1^o(X)$ to $H \cong \Z^2$ is contained
in the cyclic group $K$, the image $q(H)$ is an
infinite subgroup of $\pi_1^o(X)$, and hence is of finite index.
Since $q(H)$ is a finitely generated abelian group, $q(H)$ is also
autostackable.
Now Proposition~\ref{prop:closuresupergp} shows that $\pi_1^o(X)$ is
autostackable relative to $q(H)$.
\end{proof}

We are now ready to prove our theorem for compact 3-manifolds
with incompressible toral boundary.
Note that the proof is
very direct, and produces an autostackable structure that can, in
theory, be computed using software. 
This is in sharp contrast to the proof in~\cite{wordprocessing}
of the existence of an automatic structure on the
fundamental group of a 3-manifold with no \emph{Nil} or \emph{Sol}
pieces in its prime decomposition, which gives an automatic structure
which would be difficult to explicitly produce.

\begin{thm}\label{thm:3mfldauto}
Let $M$ be a compact 3-manifold with incompressible toral boundary.
Then $\pi_1(M)$ is autostackable. In particular, if $M$ is closed,
then $\pi_1(M)$ is autostackable.
\end{thm}

\begin{proof}
Let $\widetilde M$ be an orientable double cover of $M$
in the case that $M$ is not orientable; otherwise let
$\widetilde M:=M$.  Then $\pi_1(\widetilde M)$ is a finite index
subgroup of $\pi_1(M)$, and so~\cite[Theorem~3.4]{BHJ:closure}
(restated above as Proposition~\ref{prop:closuresupergp}) shows that
it suffices to prove that $\pi_1(\widetilde M)$ is autostackable.
Further, $\widetilde M$ also has incompressible toral boundary.

The orientable 3-manifold $\widetilde M$ has a unique decomposition
as a connected sum of prime manifolds, $\widetilde
M=M_1\#M_2\#\cdots\#M_k$.  Then the fundamental group is the free
product $\pi_1(\widetilde
M)=\pi_1(M_1)\ast\pi_1(M_2)\ast\cdots\ast\pi_1(M_k)$. As a free
product of autostackable groups is autostackable (this is shown
in~\cite[Theorem~3.2]{BHJ:closure}, but also follows as a special
case of Theorem~\ref{thm:GoGautostack}), it suffices to show that
$\pi_1(M)$ is autostackable in the case that $M$ is a prime,
orientable, compact 3-manifold with incompressible toral boundary;
for the remainder of this proof we assume that $M$ satisfies 
these four properties.

Suppose that $M$ is also geometric.  If $M$ is closed, then
by~\cite[Corollary~1.5]{BHH:algorithms}, $\pi_1(M)$ 
is autostackable.  If $M$ is not closed, then either
$M$ is Seifert fibered, in which case autostackability
of $\pi_1(M)$ is shown in
Proposition~\ref{prop:seifertautostack}, or else
the interior of $M$ is a finite volume hyperbolic 3-manifold,
and so the last sentence of Section~\ref{sub:3mfld}
together with Corollary~\ref{prop:relhypautostack} 
show that $\pi_1(M)$ is
autostackable.

On the other hand, if $M$ is not geometric, then $M$ admits a JSJ
decomposition into finitely many compact Seifert fibered and
hyperbolic pieces $\{M_v\}_{v \in V}$ also with incompressible toral
boundary. Then $\pi_1(M)$ is the fundamental group of a graph of
groups on a finite connected graph $\Lambda$ with vertex set $V$,
satisfying the property that for each $v \in V$ the vertex group is
$\pi_1(M_v)$, and for each directed edge $e$ in $\Lambda$ the edge
group $G_e \cong \Z^2$ maps via the homomorphism $\homom_e$ to the
image of the fundamental group of an incompressible torus $T_e$ in
the boundary of $M_\vend(e)$; that is, $\homom_e(G_e)=\pi_1(T_e)$.

Let $v \in V$ and let $T$ be an  incompressible torus in the
boundary of $M_v$.
In the case that $M_v$ is Seifert-fibered,
Proposition~\ref{prop:seifertautostack} shows that
$\pi_1(M_v)$ is autostackable {\rsp} $\pi_1(T)$.
In the case that the interior of $M_v$ is hyperbolic,
the fundamental group $\pi_1(M_v)$ is hyperbolic relative to
a (finite) collection of peripheral ($\Z^2$) subgroups
corresponding to the boundary components of $M_v$~\cite[Theorem~5.1]{Farb:relhyp},
and so Corollary~\ref{prop:relhypautostack} shows that
$\pi_1(M_v)$ is autostackable {\rsp} $\pi_1(T)$
in this case as well.

Therefore, by Theorem~\ref{thm:GoGautostack}, $\pi_1(M)$ is
autostackable.
\end{proof}

As a historical note, we remark that
it is a consequence of Theorem~\ref{thm:3mfldauto}
that all closed 3-manifold groups satisfy
the tame combability condition
of Mihalik and Tschantz~\cite{mihaliktschantz},
since every stackable group is tame combable~\cite{BHtame}.
Mihalik and Tschantz show that if $M$ is a
closed irreducible 3-manifold and $\pi_1(M)$
is infinite and tame combable, then 
$M$ has universal cover
homeomorphic to $\mathbb{R}^3$; 
tame combability was introduced in part
to establish a conjecture that all closed
irreducible 3-manifolds with
infinite fundamental group have universal cover
$\mathbb{R}^3$.  While Perelman's~\cite{morgantian} subsequent
proof of the geometrization
theorem has proven this conjecture (and the proof
of Theorem~\ref{thm:3mfldauto} relies on geometrization),
the proof of tame combability for all
closed 3-manifold groups shows the validity of
Mihalik and Tschantz's earlier approach.


\end{document}